\documentclass[11pt,a4paper,twoside]{article}

\usepackage[utf8]{inputenc}
\usepackage[T1]{fontenc}

\usepackage{hyperref}
\usepackage{lmodern}
\usepackage{amssymb,amsfonts,amsthm,mathtools,mathrsfs,pifont,bigints}
\usepackage{enumitem}
\usepackage{fancyhdr}
\usepackage{ifthen}

\newtheorem{theorem}{Theorem}[section]
\newtheorem{lemma}[theorem]{Lemma}
\newtheorem{prop}[theorem]{Proposition}
\newtheorem{corollary}[theorem]{Corollary}
\theoremstyle{remark}

\theoremstyle{definition}

\usepackage{tikz}

\usetikzlibrary{calc}
\usetikzlibrary{math}
\usetikzlibrary{patterns}
\usetikzlibrary{arrows.meta,decorations.shapes}

\definecolor{lightgray}{RGB}{211,211,211}
\definecolor{lightgreen}{RGB}{144, 238, 144}
\definecolor{llightgray}{RGB}{237,237,237}
\definecolor{lightyellow}{RGB}{205,201,165}
\definecolor{lightskincol}{RGB}{255,229,200}
\definecolor{skincol}{RGB}{255,195,170}
\definecolor{darkskincol}{RGB}{240,184,160}
\definecolor{skinpinkcol}{RGB}{255,204,203}
\definecolor{darkkhaki}{RGB}{204,204,0}
\definecolor{lightorange}{RGB}{255,186,102}
\definecolor{darkorange}{RGB}{255,140,0}
\definecolor{mediumturquoise}{RGB}{72, 209, 204}

\usepackage[left=3.2cm,right=3.2cm]{geometry}

\newcommand{\abs}[1]{\left| #1 \right|}
\newcommand{\norm}[1]{\left\lVert #1 \right\rVert}
\newcommand{\scp}[1]{\left\langle #1 \right\rangle}
\newcommand{\set}[1]{\left\lbrace #1\right\rbrace}

\newcommand{\B}{\mathbb{B}}
\newcommand{\D}{\mathcal{D}}
\newcommand{\Hilbert}{\mathcal{H}}
\newcommand{\K}{\mathcal{K}}
\newcommand{\M}{\mathcal{M}}

\newcommand{\Radon}{\mathcal{R}}

\newcommand{\restr}[2]{\left. #1 \right|_{#2}}
\newcommand{\vol}{\mathrm{vol}}
\renewcommand{\d}[1]{d #1\,}

\newcommand{\C}{\mathbb{C}}
\newcommand{\N}{\mathbb{N}}
\newcommand{\R}{\mathbb{R}}
\newcommand{\Sp}{\mathbb{S}}

\newcommand{\supp}[1]{\mathrm{supp}(#1)}
\newcommand{\nd}{\partial_{\nu}}

\usepackage{soul}
\usepackage{xcolor}
\colorlet{lred}{red!40}
\colorlet{lgreen}{green!40}
\colorlet{lblue}{blue!40}

\numberwithin{equation}{section}
\numberwithin{theorem}{section}

\allowdisplaybreaks

\ifpdf
\hypersetup{
  pdftitle={Recovering the Initial Data of the Wave Equation from Neumann Traces},
  pdfauthor={F. Dreier, and M. Haltmeier}
}
\fi

\begin{document}
	\title{{\fontsize{12}{15}\bfseries\uppercase{Recovering the Initial Data of the Wave Equation from Neumann Traces}}\thanks{\textbf{Funding}: This work has been supported by the Austrian Science Fund (FWF), project P 30747-N32.}}
	
	\author{Florian Dreier\thanks{Department of Mathematics, University of Innsbruck, Technikerstraße 13, A-6020 Innsbruck, Austria (Florian.Dreier@uibk.ac.at, Markus.Haltmeier@uibk.ac.at).}
\and and \and Markus Haltmeier\footnotemark[2]}

	\date{}
	
	\maketitle
	\begin{abstract}
  We study the problem of recovering the initial data  $(f, 0)$  of  the standard wave equation from the Neumann trace (the normal derivative) of the solution on the boundary of convex domains in arbitrary spatial dimension. Among others, this problem is relevant for tomographic image reconstruction including photoacoustic tomography. We establish explicit inversion formulas of the back-projection type that recover the initial data up to an additive  term defined by a smoothing integral operator. In the case that the boundary of the domain is an ellipsoid, the  integral operator vanishes, and  hence we obtain an analytic formula for recovering the initial data from Neumann traces of the wave equation on ellipsoids.
  

\end{abstract}

\section{Introduction}
	The problem of determining the initial data of the wave equation from indirect observations arises in various practical applications. Well-known examples include  photoacoustic tomography, ultrasound tomography, SONAR or seismic imaging (see, for example, \cite{BleCohJoh13,KucKun11,LidDurLitHua09,NatWub95,poudel2019survey,QuiRieSch11,rosenthal2013acoustic,WaHu12,WaPaKuXiStWa03}). These applications are often well modeled by the  standard wave equation
	\begin{equation}
		\label{eq:waveeq}
		\begin{aligned}
	    	(\partial_t^2-\Delta)u(x,t)&=0 &\quad &\text{for } (x,t) \in \R^n \times (0,\infty),\\
	    	u(x,0)&=f(x)&\quad &\text{for } x \in \R^n,&\\
	    	(\partial_tu)(x,0)&=g(x)&\quad &\text{for } x \in \R^n \,,
	    \end{aligned}
	\end{equation}
where $(f,g)$ with $f,g\colon\R^n\to\R$ are the initial data,  $\Delta$ denotes the Laplacian in the spatial component $x \in \R^n$, $\partial_t$ is the partial derivative with respect to the time $t>0$, and $n\in\N$ with $n\geq 2$ denotes the spatial dimension.

For example, in photoacoustic tomography (PAT), $g$ is the zero function and $f \in C_c^\infty(\Omega)$ a smooth function  with compact support in a bounded domain $\Omega \subset \R^n$ modeling the initial pressure distribution.
The corresponding inverse problem of  PAT  is to determine the initial data $(f,0)$ in \eqref{eq:waveeq}  from  data measured on the boundary of $\Omega$. Throughout this paper we will consider the inverse problem of PAT, where we assume  given data in  the form of Neumann traces described below. Note that the assumption of $f$ being infinitely smooth is made for the sake of simplicity only. Extensions to the case of initial data in $L^2$ based on regularity results for the wave equation \cite{LasLioTri86} seem possible, but are beyond the scope of this paper.

\subsection{Inversion from Dirichlet and Neumann traces}

Most image reconstruction methods in PAT assume that the measured data consists of the Dirichlet trace $\restr{u}{\partial\Omega\times(0,\infty)}$ on $\partial\Omega$. However, as pointed out in \cite{DreHal20,Finch2005,wang2007boundary}, PAT measurements are often more accurately modeled by a linear combination of the Dirichlet trace and the normal derivative
	\begin{equation}
		\label{eq:mixtrace}
		u_{a,b}(x,t) = a u(x,t)
		+b\nd u (x,t) \quad  \text{ for } (x,t) \in  \partial\Omega\times(0,\infty) \,,
	\end{equation}
	which we refer to as the mixed trace on $\partial\Omega$; here $a, b\geq 0$ are constants. See also \cite{acosta2019well,acosta2020solvability,cox2007frequency,paltauf2009characterization,WisPleRosNtz18} for modeling and analysis of detectors characteristics in PAT. Note that in practice neither the Dirichlet trace nor the Neumann trace are actually measured. Instead, an indirect auxiliary quantity for the pressure is measured, which depends on actual transducer mechanics. However, typical detectors exhibit a directivity effect as well as an increasing frequency response \cite{cox2007frequency,paltauf2009characterization} below the resonant frequency. Both effects are included in the model \eqref{eq:mixtrace}, which therefore represents a first step for deriving inversion formulas for realistic detector designs. Modeling the actual transducer mechanism is beyond the scope of this work. Measurements of the form  \eqref{eq:mixtrace} with $a=0$ correspond to the Neumann trace and with $b=0$ to the Dirichlet trace on $\partial\Omega$.  Beside the type of data measured in PAT, we observe from \eqref{eq:mixtrace} that the inversion of the wave equation depends on the domain  $\Omega$ as well.
	
	In the last twenty years, plenty of results for the inversion of wave equation from Dirichlet measurements have been derived. In particular, exact inversion formulas for Dirichlet data on planar surfaces \cite{And88,Bel09,BukKar78,Faw85,NarRak10,NorLin81,xu2005universal}, cylinders and spheres \cite{FinHalRak07,FinPatRak04,Kun07,Ngu09,NorLin81,xu2005universal}, ellipses \cite{AnsFibMadSey12,Hal13,Hal14,Nat12,Pal14,Sal14},  quadric hypersurfaces  \cite{HalSer15a,HalSer15b,nguyen2014reconstruction}, certain polygons and polyhedra \cite{Kun11} together with boundaries of corner-like domains \cite{Kun15} have been developed. Theoretical results for Dirichlet traces on bounded open acquisition surfaces and on a variety of other geometries which yield exact inversion can be found in \cite{DoKun18,Pal14}. However, few theoretical results for Neumann as well as for mixed traces are known. To the best of our knowledge, the only results for this inverse problem are presented in \cite{Finch2005}, where an inversion formula for a sphere in $\R^3$ is given, \cite{zangerl2018photoacoustic}, where a series inversion formula for spheres in arbitrary dimension has been established, and \cite{DreHal20}, where an exact reconstruction formula of a so-called back-projection type has been provided for an ellipse in the case of two spatial variables.
	
	In this paper, we study the inverse problem in PAT of recovering the initial data $(f,0)$ from the Neumann trace on  $\Omega$ which extend the results  of \cite{DreHal20} to arbitrary spatial dimension.  We separately treat the case of even and odd dimensions which are notably different from each other. Note that the presented results are closely related to the results of \cite{Hal14}, where corresponding formulas have been derived for data given for Dirichlet traces. While several derivations in \cite{Hal14} are based on the distributional calculus, all results in the present paper are derived using classical analysis as in \cite{DreHal20}.

	\subsection{Outline}

	In the present paper we study the problem of recovering the initial data $f\in C_c^\infty(\Omega)$ in \eqref{eq:waveeq} with $g=0$ from Neumann measurements in arbitrary dimension. We provide inversion formulas for convex domains $\Omega\subset\R^n$ with smooth boundary that are exact up to a smoothing integral operator. Moreover we derive an exact inversion  formula for Neumann traces on ellipsoids. The inversion formulas imply that the problem of recovering the initial data from Neumann traces on ellipsoids is uniquely solvable. Note, however, that for more general domains uniqueness is still an open and interesting problem. The starting point of our results is an integral identity presented in subsection \ref{subsec:integralidentity}. Based on this identity, we derive our explicit inversion  formulas  for even dimensions in section  \ref{sec:formulaevendim}  and for odd dimensions in section \ref{sec:formulaodddim}.  The paper ends with some conclusions in section \ref{sec:conclusion}.

	\section{Notation and preliminary results}
\subsection{Notation}

In the whole article, we suppose that $\Omega\subset\R^n$ is a convex domain with smooth boundary $\partial\Omega$. For a function $f\in C_c^\infty(\Omega)$ we denote by \[\M f\colon \R^n\times (0,\infty)\to \R\colon (x,r)\mapsto \frac{1}{\sigma(\partial\B^n(x,r))}\int_{\partial\B^n(x,r)} f(y) \d{\sigma(y)}\] the spherical mean operator of $f$, where $\B^n(x,r)$ denotes the open ball with center $x\in\R^n$ and radius $r>0$, and $\sigma$ the standard volume measure on manifolds. For brevity, we set $\Sp^{n-1}\coloneqq \partial\B^n(0,1)$ as the unit sphere in $\R^n$.
	
	We also use the notation \[\Radon f\colon \Sp^{n-1}\times \R\to \R\colon (\theta,s)\mapsto \int_{E(\theta,s)} f(y)\d{\sigma(y)}\] for the Radon transform of $f$, where the smooth manifold\\$E(\theta,s)\coloneqq\set{x\in\R^n\mid \scp{x,\theta}=s}$ is defined as the $n-1$-dimensional hyperplane with normal vector $\theta\in\Sp^{n-1}$ and oriented distance $s\in\R$.
	
	The Hilbert transform of a function $\varphi\colon \Sp^{n-1}\times\R\to\R$ in the second variable is defined as \[\Hilbert_2 \varphi\colon \Sp^{n-1}\times\R\to\R\colon (\theta,s)\mapsto \frac{1}{\pi}\lim_{\varepsilon\searrow 0}\int_{\R\setminus(s-\varepsilon,s+\varepsilon)}\frac{\varphi(\theta,t)}{s-t}\d{t},\] provided the integral for every $\varepsilon>0$ and the limit exist. If the function $\varphi(\theta,\cdot)\colon\R\to\R$ is differentiable for some $\theta\in\Sp^{n-1}$, then we denote its derivative by $\partial_2 \varphi(\theta,\cdot)$.
	
	Similarly, we denote by \[(g\ast_2 h)(\theta,s)\coloneqq (g(\theta,\cdot)\ast h(\theta,\cdot))(s)=\int_\R g(\theta,s-x)h(\theta,x)\d{x}\]
	the convolution of $g$ with $h$ in the second argument, where $g,h\colon \Sp^{n-1}\times\R\to\R$ and $(\theta,s)\in\Sp^{n-1}\times\R$.
	
\subsection{Approximation to the identity}

In the first auxiliary technical result, we present an approximation to the identity in $\R^n$, whose Radon transform is again an approximation to the identity in $\R$ for a fixed angle $\theta\in\Sp^{n-1}$.

	\begin{lemma}
		\label{lem:radonmollifier}
		For a positive integer $\mu>0$ let
		\begin{equation*}
			\psi_{\mu}\colon\R^n\to\R\colon x\mapsto\frac{1}{a}\begin{cases}
				(1-\norm{x}^2)^\mu,\quad &\norm{x}\leq 1\\
				0,\quad &\text{otherwise},
		\end{cases}
		\end{equation*}
		where $a\coloneqq\frac{\pi^{\frac{n}{2}}\Gamma(\mu+1)}{\Gamma(\frac{n}{2}+\mu+1)}$ and $\Gamma\colon\C\setminus(-\N)\to\C$ is the gamma function. Then the family $(\psi_{\mu,\varepsilon})_{\varepsilon>0}$ with
		\begin{equation*}
			\psi_{\mu,\varepsilon}(x)\coloneqq \varepsilon^{-n}\psi_{\mu}\left(\frac{x}{\varepsilon}\right)\quad \text{for $x\in\R^n$ and $\varepsilon>0$}
		\end{equation*}
		is an approximation to the identity in $\R^n$ as well as $(\Radon\psi_{\mu,\varepsilon}(\theta,\cdot))_{\varepsilon>0}$ in $\R$ for every $\theta\in\Sp^{n-1}$. Furthermore, the Radon transform of $\psi_{\mu,\varepsilon}$ is given by
		\begin{equation}
			\label{eq:radonmollifier}
			\Radon\psi_{\mu,\varepsilon}(\theta,s)=\frac{\Gamma\left(\frac{n}{2}+\mu+1\right)}{\varepsilon\sqrt{\pi}\Gamma\left(\frac{n-1}{2}+\mu+1\right)}\left(1-\frac{s^2}{\varepsilon^2}\right)^{\frac{n-3}{2}+\mu+1}
		\end{equation}
		for $(\theta,s)\in\Sp^{n-1}\times(-\varepsilon,\varepsilon)$ and $\Radon\psi_{\mu,\varepsilon}(\theta,s)=0$ for $\abs{s}\geq \varepsilon$.
	\end{lemma}

\subsection{Solution of the wave equation}

Next, we recall a well-known solution formula for the wave equation \eqref{eq:waveeq} and deduce other representations which are used for derivation of our general inversion formula in section \ref{sec:formulaevendim} and \ref{sec:formulaodddim}.
	From  \cite{Eva10} we know that the solution $u\colon \R^n\times(0,\infty)\to\R$ of the wave equation with initial data $f,g\in C_c^\infty(\Omega)$ is given by the representation formula
	\begin{equation}
		\label{eq:solwaveeqeven}
		\begin{aligned}
		u(x,t)=\frac{1}{\gamma_n}&\Bigg[\partial_t\left(\frac{1}{t}\partial_t\right)^{\frac{n-2}{2}}\left(\frac{t^n}{\vol(\B^n(x,t))}\int_{\B^n(x,t)}\frac{f(y)}{\sqrt{t^2-\norm{y-x}^2}}\d{y}\right)\\
		&+\left(\frac{1}{t}\partial_t\right)^{\frac{n-2}{2}}\left(\frac{t^n}{\vol(\B^n(x,t))}\int_{\B^n(x,t)}\frac{g(y)}{\sqrt{t^2-\norm{y-x}^2}}\d{y}\right)\Bigg]
		\end{aligned}
	\end{equation}
	for even $n\geq 2$ and $(x,t)\in\R^n\times(0,\infty)$, where $\gamma_n\coloneqq 2\cdot 4\cdots (n-2)\cdot n$ and $\vol(\B^n(x,t))$ denotes the volume of the $n$-dimensional ball with center $x$ and radius $t$.
	If $n\geq 3$ is odd, then the solution is given by
	\begin{equation}
		\label{eq:solwaveeqodd}
		\begin{aligned}
		u(x,t)=\frac{1}{\gamma_n}&\Bigg[\partial_t\left(\frac{1}{t}\partial_t\right)^{\frac{n-3}{2}}\left(\frac{t^{n-2}}{\sigma(\partial\B^n(x,t))}\int_{\partial\B^n(x,t)}f(y)\d{\sigma(y)}\right)\\
		&+\left(\frac{1}{t}\partial_t\right)^{\frac{n-3}{2}}\left(\frac{t^{n-2}}{\sigma(\partial\B^n(x,t))}\int_{\partial\B^n(x,t)}g(y)\d{\sigma(y)}\right)\Bigg]
		\end{aligned}
	\end{equation}
	where $\gamma_n\coloneqq 1\cdot 3\cdots (n-2)$. In terms of the spherical mean operator, we see that \eqref{eq:solwaveeqodd} can also be represented by
	\begin{equation}
		\label{eq:solwaveeqodd2}
		u(x,t)=\frac{1}{\gamma_n}\left[\partial_t\left(\frac{1}{t}\partial_t\right)^{\frac{n-3}{2}}\left(t^{n-2}\M f(x,t)\right)+\left(\frac{1}{t}\partial_t\right)^{\frac{n-3}{2}}\left(t^{n-2}\M g(x,t)\right)\right].
	\end{equation}
	
	The next lemma shows another representation of solution formula \eqref{eq:solwaveeqeven} in case of even dimensions.
	
	\begin{lemma}\label{lem:wave-even}
		Let $n\geq 2$ be an even natural number and $f,g\in C_c^\infty(\Omega)$. Then the representation formulas
		\begin{equation}
		\label{eq:solwaveeqeven2}
			\begin{aligned}
			u(x,t)=\frac{n}{\gamma_n}&\Bigg[\partial_t\left(\frac{1}{t}\partial_t\right)^{\frac{n-2}{2}}\left(\int_0^t\frac{r^{n-1}}{\sqrt{t^2-r^2}}\M f(x,r)\d{r}\right)\\
			&+\left(\frac{1}{t}\partial_t\right)^{\frac{n-2}{2}}\left(\int_0^t\frac{r^{n-1}}{\sqrt{t^2-r^2}}\M g(x,r)\d{r}\right)\Bigg]
			\end{aligned}
		\end{equation}
		and
		\begin{equation}
		\label{eq:solwaveeqeven3}
			\begin{aligned}
			u(x,t)=\frac{n}{\gamma_n}&\Bigg[\partial_t\left(\int_0^t\frac{r}{\sqrt{t^2-r^2}}\left(\frac{1}{r}\partial_r\right)^{\frac{n-2}{2}}\left(r^{n-2}\M f(x,r)\right)\d{r}\right)\\
			&+\left(\int_0^t\frac{r}{\sqrt{t^2-r^2}}\left(\frac{1}{r}\partial_r\right)^{\frac{n-2}{2}}\left(r^{n-2}\M g(x,r)\right)\d{r}\right)\Bigg]
			\end{aligned}
		\end{equation}
		for the solution of the wave equation \eqref{eq:solwaveeqeven} with initial data $(f,g)$ hold.
	\end{lemma}
	
For the rest of this paper we denote by $u$ and $v$ the solutions of the wave equation \eqref{eq:waveeq} with initial data $(f,0)$ and $(0,g)$, respectively, whenever $f,g\in C_c^\infty(\Omega)$ are given functions.
	\subsection{Key integral identity}\label{subsec:integralidentity}
	Beside the representation formulas \eqref{eq:solwaveeqodd2} and \eqref{eq:solwaveeqeven3} for the solution of the wave equation, a main ingredient for the derivation of our inversion formulas is the following integral identity for solutions of the wave equation.
	
	\begin{prop}[Integral identity for the wave equation]
		\label{prop:main-id}
		Let $f,g\in C_c^\infty(\Omega)$. Then the following identity holds:
		\begin{equation}
			\label{eq:intidentity}
			\begin{aligned}
			\int_\Omega f(x)g(x)\,dx&=2\int_{\partial\Omega}\int_0^\infty v(x,t)\nd u(x,t)\d{t}\d{\sigma(x)}\\
			&\quad +\int_\Omega \int_0^\infty \Delta(uv)(x,t)\d{t}\d{x}.
			\end{aligned}
		\end{equation}
	\end{prop}
	
	\begin{proof}
		Let $x\in\Omega$ be a fixed point in the domain. Since $u$ and $v$ are solutions of the wave equation with initial data $(f,0)$ and $(0,g)$, respectively, and decay to zero as $t$ goes to infinity (see solution formulas \eqref{eq:solwaveeqeven}, \eqref{eq:solwaveeqodd} and Lemma \ref{lem:identity2}) we see that application of integration by parts leads to 	$$\int_0^\infty v(x,t)\Delta u(x,t)\d{t}=-\int_0^\infty \partial_t u(x,t)\partial_tv(x,t)\d{t}.$$
		Then, one further application of integration by parts yields
		\begin{align*}
			-\int_0^\infty \partial_t u(x,t)\partial_tv(x,t)\d{t}&=\lim_{a\searrow 0} u(x,a)\partial_tv(x,a)+\int_0^\infty u(x,t)\Delta v(x,t)\d{t}\\
			&=f(x)g(x)+\int_0^\infty u(x,t)\Delta v(x,t)\d{t}.
		\end{align*}
		As can also be seen from \eqref{eq:solwaveeqeven}, \eqref{eq:solwaveeqodd} and Lemma \ref{lem:identity2}, $v\Delta u - u\Delta v$ is integrable on $\Omega\times(0,\infty)$, and therefore changing the order of integration gives
		\begin{equation*}
			\int_\Omega f(x)g(x)\d{x}=\int_0^\infty\int_\Omega \left(v(x,t)\Delta u(x,t)-u(x,t)\Delta v(x,t)\right)\d{x}\d{t}.
		\end{equation*}
		Next, from Green's second formula we conclude $$\int_\Omega f(x)g(x)\d{x}=\int_0^\infty\int_{\partial\Omega} \left( v(x,t)\nd u(x,t)-u(x,t)\nd v(x,t)\right) \d{\sigma(x)}\d{t}.$$ Since $$\nd v(x,t)=\scp{\nabla v(x,t),\nu(x)}\ \text{and}\ u(x,t)\nabla v(x,t)=\nabla(uv)(x,t)-v(x,t)\nabla u(x,t),$$ where $\scp{\cdot,\cdot}$ denotes the standard dot product of two vectors in $\R^n$, we see that
		\begin{align*}
			\int_\Omega f(x)g(x)\d{x}=&2\int_0^\infty\int_{\partial\Omega} v(x,t)\nd u(x,t)\d{\sigma(x)}\d{t}\\
			&+\int_0^\infty\int_{\partial\Omega} \scp{\nabla(uv)(x,t),\nu(x)}\d{\sigma(x)}\d{t}.
		\end{align*}
		In the last step, we use the divergence theorem in the second inner integral and change the order of integration afterwards to obtain the desired integral identity.
	\end{proof}
	Using the above auxiliary results, we are now ready to prove the main results of this paper. In the next section, we present our new results for the even-dimensional case.

\section{Inversion in even dimension}
\label{sec:formulaevendim}

\subsection{Statement of the inversion formula}

The following theorem is our new result for even dimensions.

	\begin{theorem}[Inversion formula in even dimension]
			\label{thm:explicitformeven}
		Let $n\geq 2$ be an even number, $\Omega\subset \R^n$ be a bounded convex domain with smooth boundary and $f\in C_c^\infty(\Omega)$. Then, for every $x\in\Omega$, we have
		\begin{equation}
			\label{eq:explictformeven}
			f(x)=\frac{1}{2^{\frac{n-2}{2}}\pi^{\frac{n}{2}}}(-1)^{\frac{n-2}{2}}\int_{\partial\Omega}\int_{\norm{x-y}}^\infty \frac{\left(\partial_t t^{-1} \right)^{\frac{n-2}{2}}\nd u(y,t)}{\sqrt{t^2-\norm{x-y}^2}}\d{t}\d{\sigma(y)}+\K_\Omega f(x),
		\end{equation}
		where
		\begin{equation*}
			\K_\Omega f (x)\coloneqq \frac{(-1)^{\frac{n-2}{2}}}{2^{n+1}\pi^{n-1}}\int_\Omega f(y)\frac{\left(\partial_2^n\Hilbert_2\Radon \chi_\Omega\right)\left(\tilde{n}(x,y),\tilde{s}(x,y)\right)}{\norm{x-y}^{n-1}}\d{y}
		\end{equation*}
		and $\tilde{n}(x,y)\coloneqq (y-x)/\norm{y-x}$, $\tilde{s}(x,y)\coloneqq (\norm{y}^2-\norm{x}^2)/(2\norm{y-x})$ for $y\in\Omega$ with $y\neq x$.
	\end{theorem}
	
\begin{proof}
The proof will be given in subsection \ref{sec:proof-even}.
\end{proof}

As a consequence of  Theorem~\ref{thm:explicitformeven} we have the following exact inversion  formula for the wave equation from Neumann traces for the case that $\Omega$ is an elliptical domain.
    	
		\begin{corollary}[Exact inversion formula for ellipsoids in even dimension]
		\label{cor:exactformeven}
		Let $n\geq 2$ be an even number,  $\Omega \subset \R^n$ be an open domain, such that $\partial\Omega$ is an ellipsoid and $f\in C_c^\infty(\Omega)$. Then, for every $x\in\Omega$, we have
		\begin{equation}
			\label{eq:exactformeven}
			f(x)=\frac{1}{2^{\frac{n-2}{2}}\pi^{\frac{n}{2}}}(-1)^{\frac{n-2}{2}}\int_{\partial\Omega}\int_{\norm{x-y}}^\infty \frac{\left(\partial_tt^{-1}\right)^{\frac{n-2}{2}}\nd u(y,t)}{\sqrt{t^2-\norm{x-y}^2}}\d{t}\d{\sigma(y)}.
		\end{equation}
	\end{corollary}
	\begin{proof}
		Taking into account Theorem \ref{thm:explicitformeven}, it remains to show that $\K_\Omega f(x)=0$ for $x\in\Omega$. This is because of $\partial_2^n\Hilbert_2\Radon\chi_\Omega=0$ for the special case where $\Omega$ is an elliptical domain as shown in \cite{Hal14}.
	\end{proof}

	The proof of  Theorem \ref{thm:explicitformeven} requires some preparation. The reconstruction formula \eqref{eq:explictformeven} consists of two terms, where the first term contains the Neumann trace of $u \colon \R^n\times(0,\infty)\to\R$ on $\partial\Omega\times(0,\infty)$ and the second term is an integral operator depending on the initial data $f \in C_c^\infty(\Omega)$. Both terms are transformations of the two terms appearing on the right-hand side of the integral identity \eqref{eq:intidentity}. We will derive the inversion formula by manipulating both terms in \eqref{eq:intidentity} separately.

	\subsection{Manipulation of the boundary term}

	The first term on the right hand side in \eqref{eq:intidentity} can be computed as follows:
	\begin{prop}
		\label{lem:transfirsttermeven}
		Let $n\geq 2$ be an even natural number and $f,g\in C_c^\infty(\Omega)$. Then the identity
		\begin{multline*}
			\int_0^\infty v(x,t)\nd u(x,t) \d{t}\\
			=\frac{1}{\omega_n\gamma_n}(-1)^{\frac{n-2}{2}}\int_0^\infty \bigg(\left(\partial_t \frac{1}{t}\right)^{\frac{n-2}{2}}\nd u(x,t)\bigg)\int_{\B^n(x,t)}\frac{g(y)}{\sqrt{t^2-\norm{y-x}^2}}\d{y}\d{t}
		\end{multline*}
		holds for every $x\in\partial\Omega$.
	\end{prop}
	For the proof of this proposition, we need the following lemma.
	\begin{lemma}
			\label{lem:identity2}
			Let $n\geq 2$ be a natural number and $g\in C_c^\infty(\Omega)$. Then, for every $k\in\N$ and $(x,t)\in\R^n\times(0,\infty)$ the identity
			\begin{align*}
				\left(\frac{1}{t}\partial_t\right)^k\int_{\B^n(x,t)}&\frac{g(y)}{\sqrt{t^2-\norm{x-y}^2}}\d{y}\\
				&=\sum_{l=0}^k c_{k,l}^{(n)}t^{n-(2k+1-l)}\int_{\B^n(0,1)}\frac{\sum_{i\in\set{1,\ldots,n}^k}\partial^ig(x+ty)y^i}{\sqrt{1-\norm{y}^2}}\d{y}\\
				&=\sum_{l=0}^k c_{k,l}^{(n)}t^{-(3k-l)}\int_{\B^n(0,t)}\frac{\sum_{i\in\set{1,\ldots,n}^k}\partial^ig(x+y)y^i}{\sqrt{t^2-\norm{y}^2}}\d{y},
			\end{align*}
			holds, where $\partial^i\coloneqq \partial_{i_1}\ldots\partial_{i_k}$, $y^i\coloneqq y_{i_1}\cdot\ldots\cdot y_{i_k}$ and the coefficients are recursively defined by $c_{0,0}^{(n)}\coloneqq 1$, $c_{1,0}^{(n)}\coloneqq n-1$, $c_{1,1}^{(n)}\coloneqq 1$, $c_{\tilde{k},0}^{(n)}\coloneqq c_{\tilde{k}-1,0}^{(n)}(n-(2(\tilde{k}-1)+1))$, $c_{\tilde{k},\tilde{k}}^{(n)}\coloneqq 1$ and $c_{\tilde{k},l}^{(n)}\coloneqq c_{\tilde{k}-1,l-1}^{(n)}+c_{\tilde{k}-1,l}^{(n)}(n-(2(\tilde{k}-1)-(l-1))$ for all $\tilde{k}\in\set{2,\ldots,k}$ and $l\in\{1,\ldots,\tilde{k}-1\}$.
		\end{lemma}
		\begin{proof}
			\begin{enumerate}[wide=\parindent,label=(\roman*)]
				\item We start the proof by showing
				\begin{equation}
					\label{eq:identity2}
					\begin{aligned}
					\left(\frac{1}{t}\partial_t\right)^k\int_{\B^n(x,t)} &\frac{g(y)}{\sqrt{t^2-\norm{x-y}^2}}\d{y}\\
					&=\sum_{l=0}^k c_{k,l}^{(n)}t^{n-(2k+1-l)}\partial_t^l\int_{\B^n(0,1)}\frac{g(x+ty)}{\sqrt{1-\norm{y}^2}}\d{y}.
					\end{aligned}
				\end{equation}
				The case $k=0$ follows from integration by substitution. Similarly, for $k=1$ we derive
				\begin{align*}
					\frac{1}{t}\partial_t&\int_{\B^n(x,t)}\frac{g(y)}{\sqrt{t^2-\norm{x-y}^2}}\d{y}\\
					&=\frac{1}{t}\partial_tt^{n-1}\int_{\B^n(0,1)}\frac{g(x+ty)}{\sqrt{1-\norm{y}^2}}\d{y}\\
					&=(n-1)t^{n-3}\int_{\B^n(0,1)}\frac{g(x+ty)}{\sqrt{1-\norm{y}^2}}\d{y}+t^{n-2}\partial_t\int_{\B^n(0,1)}\frac{g(x+ty)}{\sqrt{1-\norm{y}^2}}\d{y}.
				\end{align*}
				Now, suppose that \eqref{eq:identity2} holds for any value $k\geq 1$. Hence, our assumption and application of the product rule yield the relation
				\begin{align*}
					\left(\frac{1}{t}\partial_t\right)^{k+1}&\int_{\B^n(x,t)}\frac{g(y)}{\sqrt{t^2-\norm{x-y}^2}}\d{y}\\
					&=\frac{1}{t}\partial_t\left(\sum_{l=0}^k c_{k,l}^{(n)}t^{n-(2k+1-l)}\partial_t^l\int_{\B^n(0,1)}\frac{g(x+ty)}{\sqrt{1-\norm{y}^2}}\d{y}\right)\\
					&=\sum_{l=0}^k c_{k,l}^{(n)}(n-(2k+1-l))t^{n-(2(k+1)+1-l)}\partial_t^{l}\int_{\B^n(0,1)}\frac{g(x+ty)}{\sqrt{1-\norm{y}^2}}\d{y}\\
					&\quad+c_{k,l}^{(n)}t^{n-(2(k+1)+1-(l+1))}\partial_t^{l+1}\int_{\B^n(0,1)}\frac{g(x+ty)}{\sqrt{1-\norm{y}^2}}\d{y},
				\end{align*}
				where the last sum can be written as
				\begin{align*}
					&c_{k,0}^{(n)}(n-(2k+1))t^{n-(2(k+1)+1)}\int_{\B^n(0,1)}\frac{g(x+ty)}{\sqrt{1-\norm{y}^2}}\d{y}\\
					&\ +\sum_{l=1}^k\left(c_{k,l}^{(n)}(n-(2k+1-l))+c_{k,l-1}^{(n)}\right)t^{n-(2(k+1)+1-l)}\partial_t^{l}\int_{\B^n(0,1)}\frac{g(x+ty)}{\sqrt{1-\norm{y}^2}}\d{y}\\
					&\ +c_{k,k}^{(n)}t^{n-(2(k+1)+1-(k+1))}\partial_t^{k+1}\int_{\B^n(0,1)}\frac{g(x+ty)}{\sqrt{1-\norm{y}^2}}\d{y}.
				\end{align*}
				Comparing the above coefficients with the coefficients defined in the lemma shows \eqref{eq:identity2}.
				\item Differentiating under the integral sign and using the chain rule lead to
				\begin{equation*}
					\partial_t^k\int_{\B^n(0,1)}\frac{g(x+ty)}{\sqrt{1-\norm{y}^2}}\d{y}=\int_{\B^n(0,1)}\frac{\sum_{i\in\set{1,\ldots,n}^k}\partial^ig(x+ty)y^i}{\sqrt{1-\norm{y}^2}}\d{y}.
				\end{equation*}
				Then, substituting $y$ with $\frac{y}{t}$ yields the relation
				\begin{equation*}
					\partial_t^k\int_{\B^n(0,1)}\frac{g(x+ty)}{\sqrt{1-\norm{y}^2}}\d{y}=\frac{1}{t^{n+k-1}}\int_{\B^n(0,t)}\frac{\sum_{i\in\set{1,\ldots,n}^k}\partial^ig(x+y)y^i}{\sqrt{1-\norm{y}^2}}\d{y},
				\end{equation*}
				which shows together with \eqref{eq:identity2} the desired identity.\qedhere
			\end{enumerate}
		\end{proof}
		From Lemma \ref{lem:identity2} we immediately obtain the following corollary.
		\begin{corollary}
			Under the assumptions of Proposition \ref{lem:transfirsttermeven} we have for all  even $n\geq 4$ and $0\leq k\leq \frac{n-4}{2}$
			\begin{equation}
				\label{eq:limit}
				\lim_{t\to\infty}\frac{1}{t}\bigg(\left(\partial_t \frac{1}{t}\right)^{\frac{n-4}{2}-k}\nd u(x,t)\bigg)\left(\frac{1}{t}\partial_t\right)^k\int_{\B^n(x,t)}\frac{g(y)}{\sqrt{t^2-\norm{x-y}^2}}\d{y}=0
			\end{equation}
			and
			\begin{equation}
				\label{eq:limit2}
				\lim_{t\searrow 0}\frac{1}{t}\bigg(\left(\partial_t \frac{1}{t}\right)^{\frac{n-4}{2}-k}\nd u(x,t)\bigg)\left(\frac{1}{t}\partial_t\right)^k\int_{\B^n(x,t)}\frac{g(y)}{\sqrt{t^2-\norm{x-y}^2}}\d{y}=0.
			\end{equation}
		\end{corollary}
		\begin{proof}
			\begin{enumerate}[wide=\parindent,label=(\roman*)]
				\item Since $g$ has compact support in $\Omega$, we can find $R>0$ such that $\supp{g(x+\cdot})\subset\B^n(0,R)$. Then, from the previous lemma we have for $t\geq R$
			\begin{align*}
				\Bigg\vert\left(\frac{1}{t}\partial_t\right)^k&\int_{\B^n(x,t)}\frac{g(y)}{\sqrt{t^2-\norm{x-y}^2}}\d{y}\Bigg\vert\\
				&=\Bigg\vert\sum_{l=0}^k c_{k,l}^{(n)}t^{-(3k-l)}\int_{\B^n(0,R)}\frac{\sum_{i\in\set{1,\ldots,n}^k}\partial^ig(x+y)y^i}{\sqrt{t^2-\norm{y}^2}}\d{y}\Bigg\vert\\
				&\leq C\sum_{l=0}^k c_{k,l}^{(n)}t^{-(3k-l)}\int_{\B^n(0,R)}\frac{1}{\sqrt{R^2-\norm{y}^2}}\d{y}\\
				&=C\sum_{l=0}^kc_{k,l}^{(n)}t^{-(3k-l)}n\omega_n\int_0^R\frac{r^{n-1}}{\sqrt{R^2-r^2}}\d{r}\\
				&=C\sum_{l=0}^kc_{k,l}^{(n)}t^{-(3k-l)}n\omega_n\frac{\sqrt{\pi} R^{n-1}\Gamma(\frac{n}{2})}{2\Gamma(\frac{n+1}{2})}
			\end{align*}
			where we used polar coordinates in the last integral and set $$C\coloneqq R^kn^k\max_{i\in\set{1,\ldots,n}^k}\max_{y\in\R^n}\abs{\partial^ig(y)}.$$ This implies Equality \eqref{eq:limit}.
				\item From Lemma \ref{lem:identity2} we conclude that
				\begin{multline*}
					\abs{\frac{1}{t}\bigg(\left(\partial_t \frac{1}{t}\right)^{\frac{n-4}{2}-k}\nd u(x,t)\bigg)\left(\frac{1}{t}\partial_t\right)^k\int_{\B^n(x,t)}\frac{g(y)}{\sqrt{t^2-\norm{x-y}^2}}\d{y}}\\
					\leq D\abs{\frac{1}{t}\bigg(\left(\partial_t \frac{1}{t}\right)^{\frac{n-4}{2}-k}\nd u(x,t)\bigg)}\left(\sum_{l=0}^k c_{k,l}^{(n)}t^{n-(2k+1-l)}\int_{\B^n(0,1)}\frac{1}{\sqrt{1-\norm{y}^2}}\d{y}\right)
				\end{multline*}
				where $$D\coloneqq n^k\max_{i\in\set{1,\ldots,n}^k}\max_{y\in\R^n}\abs{\partial^ig(y)}.$$ After applying the product rule on the first factor, we see that the above term consists of variables $t$ with exponents greater than or equal to $2$. This observation shows \eqref{eq:limit2}.
			\end{enumerate}
		\end{proof}
		\begin{proof}[Proof of Proposition \ref{lem:transfirsttermeven}]
			The case $n=2$ is clear. For $n\geq 4$ we prove the statement by showing
			\begin{equation*}
				%\label{eq:identity3}
				\begin{aligned}
				&\int_0^\infty v(x,t)\nd u(x,t) \d{t}\\
				&=\frac{1}{\omega_n\gamma_n}(-1)^k\int_0^\infty \bigg(\left(\partial_t \frac{1}{t}\right)^{k}\nd u(x,t)\bigg)\left(\frac{1}{t}\partial
				_t\right)^{\frac{n-2}{2}-k}\int_{\B^n(x,t)}\frac{g(y)}{\sqrt{t^2-\norm{y-x}^2}}\d{y}\d{t}
				\end{aligned}
			\end{equation*}
			for all $1\leq k\leq \frac{n-2}{2}$. For $k=1$, inserting representation formula \eqref{eq:solwaveeqeven} for $v$ and applying integration by parts yield the relation
			\begin{align*}
				\int_0^\infty &v(x,t)\nd u(x,t) \d{t}\\
				&=\frac{1}{\omega_n\gamma_n}\Bigg(\lim_{a\searrow 0}\int_a^c \frac{1}{t}\nd u(x,t)\partial_t\left(\frac{1}{t}\partial
				_t\right)^{\frac{n-4}{2}}\int_{\B^n(x,t)}\frac{g(y)}{\sqrt{t^2-\norm{y-x}^2}}\d{y}\d{t}\\
				&\quad +\lim_{b\to\infty}\int_c^b \frac{1}{t}\nd u(x,t)\partial_t\left(\frac{1}{t}\partial
				_t\right)^{\frac{n-4}{2}}\int_{\B^n(x,t)}\frac{g(y)}{\sqrt{t^2-\norm{y-x}^2}}\d{y}\d{t}\Bigg)\\
				&=-\frac{1}{\omega_n\gamma_n}\int_0^\infty \bigg(\partial_t \frac{1}{t}\nd u(x,t)\bigg)\left(\frac{1}{t}\partial
				_t\right)^{\frac{n-4}{2}}\int_{\B^n(x,t)}\frac{g(y)}{\sqrt{t^2-\norm{y-x}^2}}\d{y}\d{t}
			\end{align*}
			for some $c\in(0,\infty)$, where the boundary term is zero because of \eqref{eq:limit} and \eqref{eq:limit2}. Now, suppose that the above identity holds for any value $1\leq k<\frac{n-2}{2}$. Then, by using the same arguments as before we obtain
			\begin{align*}
				&\int_0^\infty v(x,t)\nd u(x,t) \d{t}\\
				&=\frac{(-1)^k}{\omega_n\gamma_n}\int_0^\infty \bigg(\left(\partial_t \frac{1}{t}\right)^{k}\nd u(x,t)\bigg)\left(\frac{1}{t}\partial
				_t\right)^{\frac{n-2}{2}-k}\int_{\B^n(x,t)}\frac{g(y)}{\sqrt{t^2-\norm{y-x}^2}}\d{y}\d{t}\\
				&=\frac{(-1)^{k+1}}{\omega_n\gamma_n}\int_0^\infty \bigg(\left(\partial_t \frac{1}{t}\right)^{k+1}\nd u(x,t)\bigg)\\
				&\hspace{2.5cm}\cdot \left(\frac{1}{t}\partial
				_t\right)^{\frac{n-2}{2}-(k+1)}\int_{\B^n(x,t)}\frac{g(y)}{\sqrt{t^2-\norm{y-x}^2}}\d{y}\d{t},
			\end{align*}
			which finally proves our desired transformation.
		\end{proof}
	
	\subsection{Manipulation of the interior term}
	
	%In this section we derive a relation that allows to reformulate the second term on the right hand side in \eqref{eq:intidentity} such that it can be seen to vanish for special domains $\Omega$.
	In this section we derive a relation that allows to reformulate \eqref{eq:intidentity} such that the interior term, which is the integral over the domain $\Omega$ on the right hand side, can be seen to vanish for special domains $\Omega$. This transformation requires a lot of calculations. We therefore separate the proof in several parts in order to structure the proof better for the reader.
	
	The first lemma reads as follows:
	\begin{lemma}
		\label{lem:identity1term2even}
		Let $n\geq 2$ be an even natural number and $f,g\in C_c^\infty(\Omega)$. Then, for every $x\in\Omega$ we have
		\begin{equation}
			\label{eq:identity1term2even}
			\begin{aligned}
			\int_0^\infty & u(x,t)v(x,t)\d{t}\\
			&=-\frac{n}{\gamma_n}\int_0^\infty f(x)r_2\D_{r_2}\left(r_2^{n-2}\M g(x,r_2)\right)\log r_2 \d{r_2}\\
			&\quad -\frac{n^2}{2 \gamma_n^2}\int_0^\infty\int_0^\infty \partial_{r_1}\D_{r_1}^{\frac{n-2}{2}}\left(r_1^{n-2}\M f(x,r_1)\right)r_2
			\\
			& \hspace{0.225\textwidth}
			\cdot \D_2^{\frac{n-2}{2}}\left(r_2^{n-2}\M g(x,r_2)\right)
			\log\abs{r_2^2-r_1^2}\d{r_2}\d{r_1} \,,
			\end{aligned}
		\end{equation}
		where $\D_r\coloneqq \frac{1}{r}\partial_r$.
	\end{lemma}
	\begin{proof}
		\begin{enumerate}[wide=\parindent,label=(\roman*)]
			\item \label{item:identity1term2even1}Let $T$ be a fixed positive number greater than the diameter of $\Omega$. Using representation formula \eqref{eq:solwaveeqeven3} and applying integration by parts lead to
			\begin{align*}
				\int_0^T &u(x,t)v(x,t)\d{t}\\
				&=\frac{n^2}{\gamma_n^2}\int_0^T\partial_t\left(\int_0^t\frac{r_1}{\sqrt{t^2-r_1^2}}\D_{r_1}^{\frac{n-2}{2}}\left(r_1^{n-2}\M f(x,r_1)\right)\d{r_1}\right)\\
				&\hspace{1.5cm}\cdot\int_0^t\frac{r_2}{\sqrt{t^2-r_2^2}}\D_{r_2}^{\frac{n-2}{2}}\left({r_2}^{n-2}\M g(x,r_2)\right)\d{r_2}\d{t}\\
				&=\frac{n^2}{\gamma_n^2}\int_0^T \partial_t\left(t\frac{\gamma_n}{n}f(x)+\int_0^t\sqrt{t^2-r_1^2}\partial_{r_1}\D_{r_1}^{\frac{n-2}{2}}\left(r_1^{n-2}\M f(x,r_1)\right)\d{r_1}\right)\\
				&\hspace{1.5cm}\cdot\int_0^t\frac{r_2}{\sqrt{t^2-r_2^2}}\D_{r_2}^{\frac{n-2}{2}}\left({r_2}^{n-2}\M g(x,r_2)\right)\d{r_2}\d{t}
			\end{align*}
			which, according to  Leibniz's integral rule, can be further be written as
			\begin{multline*}
				\int_0^Tf(x)v(x,t)\d{t}
				+\frac{n^2}{\gamma_n^2}\int_0^T\int_0^t\int_0^t\frac{tr_2}{\sqrt{t^2-r_1^2}\sqrt{t^2-r_2^2}} \\
				\cdot \partial_{r_1}\D_{r_1}^{\frac{n-2}{2}}\left(r_1^{n-2}\M f(x,r_1)\right)\D_{r_2}^{\frac{n-2}{2}}\left({r_2}^{n-2}\M g(x,r_2)\right)\d{r_2}\d{r_1}\d{t} \,.
			\end{multline*}
			 From $$\partial_t \log\left(\sqrt{t^2-r_2^2}+\sqrt{t^2-r_1^2}\right)= \frac{t}{\sqrt{t^2-r_1^2}\sqrt{t^2-r_2^2}}$$ we see that the right triple integral is equal to
				\begin{equation}
					\label{eq:identity1term2eventrans}
					\begin{aligned}
					\frac{n^2}{\gamma_n^2}\int_0^T&\partial_{r_1}\D_{r_1}^{\frac{n-2}{2}}\left(r_1^{n-2}\M f(x,r_1)\right)\int_0^Tr_2\D_{r_2}^{\frac{n-2}{2}}\left({r_2}^{n-2}\M g(x,r_2)\right)\\
					&\hspace{3.5cm}\cdot\int_{\max(\lbrace r_1,r_2\rbrace)}\frac{t}{\sqrt{t^2-r_1^2}\sqrt{t^2-r_2^2}}\d{t}\d{r_2}\d{r_1}\\
					&=\frac{n^2}{\gamma_n^2}\int_0^T\int_0^T\partial_{r_1}\D_{r_1}^{\frac{n-2}{2}}\left(r_1^{n-2}\M f(x,r_1)\right)r_2\D_{r_2}^{\frac{n-2}{2}}\left({r_2}^{n-2}\M g(x,r_2)\right)\\
					&\hspace{2.25cm}\cdot\log\left(\sqrt{T^2-r_2^2}+\sqrt{T^2-r_1^2}\right)\d{r_2}\d{r_1}\\
					&\quad -\frac{n^2}{2\gamma_n^2}\int_0^T\int_0^T\partial_{r_1}\D_{r_1}^{\frac{n-2}{2}}\left(r_1^{n-2}\M f(x,r_1)\right)r_2\D_{r_2}^{\frac{n-2}{2}}\left({r_2}^{n-2}\M g(x,r_2)\right)
					\\
					&\hspace{2.75cm}\cdot \log\abs{r_2^2-r_1^2}\d{r_2}\d{r_1}
					\end{aligned}
				\end{equation}
			after changing the order of the integrals.
			
			\item \label{item:identity1term2even2} We further observe that the first double integral on the right hand side in \eqref{eq:identity1term2eventrans} can be written as
			\begin{align*}
				-\frac{n}{\gamma_n}\int_0^T &f(x)r_2\D_{r_2}^{\frac{n-2}{2}}\left({r_2}^{n-2}\M g(x,r_2)\right)\log\left(T+\sqrt{T^2-r_2^2}\right)\d{r_2}\\
				&+\frac{n^2}{\gamma_n^2}\int_0^T\int_0^T\D_{r_1}^{\frac{n-2}{2}}\left(r_1^{n-2}\M f(x,r_1)\right)r_2\D_{r_2}^{\frac{n-2}{2}}\left({r_2}^{n-2}\M g(x,r_2)\right)\\
				&\hspace*{4cm}\cdot\frac{r_1}{\left(\sqrt{T^2-r_1^2}+\sqrt{T^2-r_2^2}\right)\sqrt{T^2-r_1^2}}\d{r_1}\d{r_2},
			\end{align*}
			by a further application of integration by parts, whereby the first term can be expressed as
			 \begin{multline*}
			 	-\frac{n}{\gamma_n}\int_0^\infty f(x)\int_{r_2}^T\frac{r_2\D_{r_2}^{\frac{n-2}{2}}\left({r_2}^{n-2}\M g(x,r_2)\right)}{\sqrt{t^2-r_2^2}}\d{t}\d{r_2}\\
			 	-\frac{n}{\gamma_n}\int_0^Tf(x)r_2\D_{r_2}^{\frac{n-2}{2}}\left({r_2}^{n-2}\M g(x,r_2)\right)\log(r_2)\d{r_2}.
			 \end{multline*}
			 From Fubini's theorem and solution formula \eqref{eq:solwaveeqeven3} we can conclude that this term corresponds to
			 \begin{align*}
			 	-\int_0^\infty f(x)v(x,t)\d{t}-\frac{n}{\gamma_n}\int_0^Tf(x)r_2\D_{r_2}^{\frac{n-2}{2}}\left({r_2}^{n-2}\M g(x,r_2)\right)\log(r_2)\d{r_2}.
			 \end{align*}
		\end{enumerate}
		Finally, Items \ref{item:identity1term2even1}, \ref{item:identity1term2even2} and letting $T\to\infty$ lead to identity \eqref{eq:identity1term2even}.
	\end{proof}
	Based on Lemma \ref{lem:identity1term2even}, we can now prove the following statement:
	\begin{lemma}
		\label{lem:identity2term2even}
		Under the assumptions of Lemma \ref{lem:identity1term2even} we have that
		\begin{equation*}
			%\label{eq:identity1term2even}
			\begin{aligned}
			\int_\Omega\int_0^\infty & u(x,t)v(x,t)\d{t}\d{x}\\
			&=\frac{(-1)^{\frac{n}{2}}2^{n-2}n^2}{\gamma_n^2}\lim_{m,k\to \infty}\int_{\R^n}\int_0^\infty\int_0^\infty \varphi_{m}(x)r_1^{n-1}r_2^{n-1}\M f(x,r_1)\M g(x,r_2)\\
			& \hspace{5.75cm}\cdot \Phi_{k}^{(n-1)}(r_2^2-r_1^2)\d{r_2}\d{r_1}\d{x}
			\end{aligned}
		\end{equation*}
		where $\set{\varphi_m}_{m\in\N}$ is a family of bounded and integrable functions converging pointwise to $\chi_{\Omega}$ and $\set{\Phi_k}_{k\in\N}$ a family of smooth functions converging pointwise to $\log|\cdot|$ almost everywhere. Here, we set for short $\lim_{m,k\to\infty}$ as $\lim_{m\to\infty}\lim_{k\to\infty}$.
	\end{lemma}
	\begin{proof}
		\begin{enumerate}[wide=\parindent,label=(\roman*)]
			\item \label{item:identity2term2even}We first show
			\begin{align*}
				\int_{\Omega}&\int_0^\infty\int_0^\infty \partial_{r_1}\D_{r_1}^{\frac{n-2}{2}}\left(r_1^{n-2}\M f(x,r_1)\right)r_2\D_2^{\frac{n-2}{2}}\left(r_2^{n-2}\M g(x,r_2)\right)\\
				&\hspace{1.375cm}\cdot\log\abs{r_2^2-r_1^2}\d{r_2}\d{r_1}\d{x}\\
				&=\lim_{m,k\to\infty}\Bigg[\frac{(-1)^{\frac{n}{2}}2^{\frac{n-2}{2}}\gamma_n}{n}\int_{\R^n}\int_0^\infty \varphi_m(x)f(x)r_2^{n-1}\M g(x,r_2)\Phi_k^{\left(\frac{n-2}{2}\right)}\left(r_2^2\right)\d{r_2}\d{x}\\
				&\hspace{1.75cm} + (-1)^{\frac{n-2}{2}}2^{n-1}\int_{\R^n}\int_0^\infty\int_0^\infty \varphi_m(x)r_1^{n-1}r_2^{n-1}\M f(x,r_1)\M g(x,r_2)\\
				&\hspace*{6.25cm}\cdot\Phi_k^{(n-1)}(r_2^2-r_1^2)\d{r_2}\d{r_1}\d{x}\Bigg]
			\end{align*}
			for all $x\in\Omega$. For that purpose, we apply Lebesgue's dominated convergence theorem to deduce
			\begin{equation}
			\label{eq:identity2term2even1}
			\begin{aligned}
			&\int_{\Omega}\int_0^\infty\int_0^\infty \partial_{r_1}\D_{r_1}^{\frac{n-2}{2}}\left(r_1^{n-2}\M f(x,r_1)\right)r_2\D_2^{\frac{n-2}{2}}\left(r_2^{n-2}\M g(x,r_2)\right)\\
			&\hspace{2cm}\cdot\log\abs{r_2^2-r_1^2}\d{r_2}\d{r_1}\d{x}\\
			&=\lim_{m,k\to\infty}\Bigg[\int_{\R^n}\int_0^\infty\int_0^\infty \varphi_m(x)\partial_{r_1}\D_{r_1}^{\frac{n-2}{2}}\left(r_1^{n-2}\M f(x,r_1)\right)\\
			&\hspace{3.75cm}\cdot r_2\D_2^{\frac{n-2}{2}}\left(r_2^{n-2}\M g(x,r_2)\right)\Phi_k\left(r_2^2-r_1^2\right)\d{r_2}\d{r_1}\d{x}\Bigg].
			\end{aligned}
			\end{equation}
			Furthermore, application of integration $\frac{n-2}{2}$-times with respect to $r_2$ gives us
			\begin{align*}
				\int_0^\infty\int_0^\infty & \partial_{r_1}\D_{r_1}^{\frac{n-2}{2}}\left(r_1^{n-2}\M f(x,r_1)\right)r_2\D_2^{\frac{n-2}{2}}\left(r_2^{n-2}\M g(x,r_2)\right)\Phi_k\left(r_2^2-r_1^2\right)\d{r_2}\d{r_1}\\
				&=(-2)^{\frac{n-2}{2}}\int_0^\infty\int_0^\infty \partial_{r_1}\D_{r_1}^{\frac{n-2}{2}}\left(r_1^{n-2}\M f(x,r_1)\right)r_2^{n-1}\M g(x,r_2)\\
				&\hspace{3cm}\cdot\Phi_k^{\left(\frac{n-2}{2}\right)}\left(r_2^2-r_1^2\right)\d{r_2}\d{r_1},
			\end{align*}
			where changing the order of integration and partial integration with respect $r_1$ lead then to
			\begin{align*}
				&\frac{(-1)^{\frac{n}{2}}2^{\frac{n-2}{2}}\gamma_n}{n}\int_0^\infty f(x)r_2^{n-1}\M g(x,r_2)\Phi_k^{\left(\frac{n-2}{2}\right)}\left(r_2^2\right)\d{r_2}\\
				&\quad + (-1)^{\frac{n-2}{2}}2^{\frac{n}{2}}\int_0^\infty\int_0^\infty r_1\D_{r_1}^{\frac{n-2}{2}}\left(r_1^{n-2}\M f(x,r_1)\right)r_2^{n-1}\M g(x,r_2)\\
				&\hspace{3.75cm}\cdot\Phi_k^{\left(\frac{n}{2}\right)}\left(r_2^2-r_1^2\right)\d{r_2}\d{r_1}.
			\end{align*}
			Finally, the application of integration by parts formula $\frac{n-2}{2}$-times with respect to variable $r_1$ shows together with \eqref{eq:identity2term2even1} the above identity.
			\item In the last step we apply partial integration $\frac{n-2}{2}$-times on the first term to obtain the relation
			\begin{multline*}
				\int_0^\infty \varphi_m(x)f(x)r_2^{n-1}\M g(x,r_2)\Phi_k^{\left(\frac{n-2}{2}\right)}\left(r_2^2\right)\d{r_2}\\
				=\left(-\frac{1}{2}\right)^{\frac{n-2}{2}}\int_0^\infty \varphi_m(x)f(x)r_2\D_{r_2}^{\frac{n-2}{2}}\left(r_2^{n-2} \M g(x,r_2)\right)\Phi_k\left(r_2^2\right)\d{r_2}.
			\end{multline*}
			Thus, by applying Lebesgue’s dominated convergence theorem we have
			\begin{multline*}
				\lim_{m,k\to\infty}\Bigg[\frac{(-1)^{\frac{n}{2}}2^{\frac{n-2}{2}}\gamma_n}{n}\int_{\R^n}\int_0^\infty \varphi_m(x)f(x)r_2^{n-1}\M g(x,r_2)\Phi_k^{\left(\frac{n-2}{2}\right)}\left(r_2^2\right)\d{r_2}\Bigg]\\
				=-\frac{2\gamma_n}{n}\int_\Omega\int_0^\infty f(x)r_2\D_{r_2}^{\frac{n-2}{2}}\left(r_2^{n-2} \M g(x,r_2)\right)\log r_2\d{r_2}\d{x},
			\end{multline*}
			which shows the claimed identity by using Lemma \ref{lem:identity1term2even} and \ref{item:identity2term2even}.\qedhere
		\end{enumerate}
	\end{proof}
	Now, we are ready to reshape the above double integral into our final transformation.
	\begin{prop}
		\label{lem:transsecondtermeven}
		The double integral on the left hand side in Lemma \ref{lem:identity2term2even} can be finally transformed to
		\begin{multline*}
			\int_\Omega\int_0^\infty u(x,t)v(x,t)\d{t}\d{x}\\
			=\frac{(-1)^{\frac{n-2}{2}}}{2^{n+1}\pi^{n-1}}\int_\Omega g(x)\int_\Omega f(y)\frac{\left(\partial_2^{n-2}\Hilbert_2\Radon \chi_\Omega\right)\left(\tilde{n}(x,y),\tilde{s}(x,y)\right)}{\norm{x-y}^{n-1}}\d{y}\d{x}.
		\end{multline*}
	\end{prop}
	\begin{proof}
		In the following, we use the approximation of identity $\{\psi_{\mu,\frac{1}{m}}\}_{m\in\N}$ from Lemma \ref{lem:radonmollifier} with $\mu=\frac{n}{2}+1$ and set $\varphi_m\coloneqq \chi_{\Omega}\ast \psi_{\mu,\frac{1}{m}}$ for $m\in\N$. Hence, from the properties of the convolution operator and Lemma \ref{lem:radonmollifier} we see that the family $\set{\varphi_m}_{m\in\N}$ satisfies the assumption in Lemma \ref{lem:identity2term2even} and $\varphi_m\in C_c^n(\R^n)$ as well as $\Radon\varphi_m(\theta,\cdot)\in C_c^n(\R)$ for every $\theta\in\Sp^{n-1}$.
		\begin{enumerate}[wide=\parindent,label=(\roman*)]
			\item \label{item:identity3term2even}
			In the first step of the proof we show
			\begin{multline*}
				\int_\Omega\int_0^\infty u(x,t)v(x,t)\d{t}\d{x}\\
				=\frac{(-1)^{\frac{n-2}{2}}}{2^{n+1}\pi^{n-1}}\lim_{m\to\infty}\int_\Omega g(x)\int_\Omega f(y)\frac{\partial_2^{n-2}\Hilbert_2\Radon\varphi_{m}\left(\tilde{n}(x,y),\tilde{s}(x,y)\right)}{\norm{x-y}^{n-1}}\d{y}\d{x}.
			\end{multline*}
			We observe that the right triple integral in Lemma \ref{lem:identity2term2even} inside the limit equals
			\begin{multline*}
				\int_{\R^n}\int_0^\infty\int_0^\infty \varphi_{m}(x)r_1^{n-1}r_2^{n-1}\M f(x,r_1)\M g(x,r_2)\Phi_{k}^{(n-1)}(r_2^2-r_1^2)\d{r_2}\d{r_1}\d{x}\\
				=\frac{1}{n^2\omega_n^2}\int_{\R^n}\int_\Omega\int_\Omega \varphi_{m}(x)g(y)f(z)\Phi_{k}^{(n-1)}(\norm{x-y}^2-\norm{x-z}^2)\d{z}\d{y}\d{x}
			\end{multline*}
			by using polar coordinates and the substitution rule. Furthermore, Fubini's theorem and the relation $$\norm{x-y}^2-\norm{x-z}^2=2\scp{z-y,x-(z+y)/2}$$ give us
			\begin{equation}
				\label{eq:identity3term2even1}
				\frac{1}{n^2\omega_n^2}\int_\Omega g(y)\int_\Omega f(z)\int_{\R^n}\varphi_{m}(x)\Phi_{k}^{(n-1)}(2\scp{z-y,x-(z+y)/2})\d{x}\d{z}\d{y}.
			\end{equation}
			Next, we apply the substitution rule with the diffeomorphism $$h_{y,z}\colon \R^n\to\R^n\colon (s,x_1,\ldots,x_{n-1})\mapsto s\tilde{n}(y,z)+\sum_{i=1}^{n-1}x_i \theta_i,$$ where $(\tilde{n}(y,z),\theta_1,\ldots,\theta_{n-1})$ is an orthonormal basis of $\R^n$, and Fubini's theorem on the inner integral in \eqref{eq:identity3term2even1} to obtain
			\begin{multline*}
				\int_{\R^n}\varphi_{m}(x)\Phi_{k}^{(n-1)}(2\scp{z-y,x-(z+y)/2})\d{x}\\
				=\int_{\R}\int_{\R^{n-1}}\varphi_{m}(h_{y,z}(s,x))\Phi_{k}^{(n-1)}(2\scp{z-y,h_{y,z}(s,x)-(z+y)/2})\d{x}\d{s}.
			\end{multline*}
			Since $z-y$ is orthogonal to $\theta_1,\ldots,\theta_{n-1}$ we have
			\begin{align*}
				\int_{\R^n}&\varphi_{m}(x)\Phi_{k}^{(n-1)}(2\scp{z-y,x-(z+y)/2})\d{x}\\
				&=\int_{\R}\int_{\R^{n-1}}\varphi_{m}(h_{y,z}(s,x))\Phi_{k}^{(n-1)}\left(\scp{z-y,s\tilde{n}(y,z)-(z+y)/2}\right)\d{x}\d{s}\\
				&=\int_{\R}\int_{\R^{n-1}}\varphi_{m}(h_{y,z}(s,x))\Phi_{k}^{(n-1)}\left(2\norm{z-y}(s-\tilde{s}(y,z))\right)\d{x}\d{s}\\
				&=\int_{\R}\Radon\varphi_{m}\left(\tilde{n}(y,z),s\right)\Phi_{k}^{(n-1)}\left(2\norm{z-y}(s-\tilde{s}(y,z))\right)\d{s}\\
				&=\frac{(-1)^{n-1}}{2^{n-1}\norm{z-y}^{n-1}}\int_{\R}\partial_2^{n-1}\Radon\varphi_{m}\left(\tilde{n}(y,z),s\right)\Phi_{k}\left(2\norm{z-y}(s-\tilde{s}(y,z))\right)\d{s},
			\end{align*}
			where we applied partial integration $n-1$-times in the last step. Moreover, by choosing $\set{\Phi_n}_{n\in\N}$ such that $\abs{\Phi_n}\leq \abs{\log\abs{\cdot}}$ for all $n\in\N$ we deduce from the estimate
			\begin{align*}
				\int_\Omega&\int_\Omega\int_\R\abs{\frac{g(y)f(z)}{\norm{z-y}^{n-1}}\partial_2^{n-1}\Radon\varphi_{m}\left(\tilde{n}(y,z),s\right)\Phi_{k}\left(2\norm{z-y}(s-\tilde{s}(y,z)\right)}\d{s}\d{z}\d{y}\\
				&\leq \int_\Omega\int_\Omega \frac{\abs{g(y)f(z)}}{\norm{z-y}^{n-1}}\int_\R \abs{\partial_2^{n-1}\Radon\varphi_{m}\left(\tilde{n}(y,z),s+\tilde{s}(y,z)\right)}\\
				&\hspace{4cm}\cdot\left(\abs{\log\abs{2\norm{z-y})}}+\abs{\log\abs{s}}\right)\d{s}\d{z}\d{y} < \infty,
			\end{align*}
			Lebesgue's theorem and Lemma \ref{lem:identity2term2even}
			\begin{align*}
				\int_\Omega&\int_0^\infty u(x,t)v(x,t)\d{t}\d{x}\\
				&=\frac{(-1)^{\frac{n-2}{2}}}{2\gamma_n^2\omega_n^2}\lim_{m\to \infty}\int_\Omega g(y)\int_\Omega \frac{f(z)}{\norm{z-y}^{n-1}}\\
				&\hspace{3.25cm}\cdot\int_{\R}\partial_2^{n-1}\Radon\varphi_{m}\left(\tilde{n}(y,z),s\right)\log\abs{2\norm{z-y}(s-\tilde{s}(y,z))}\d{s}\d{z}\d{y}.
			\end{align*}
			Furthermore, application of integration by parts to the inner integral yields
			\begin{align*}
				\int_{\R}\partial_2^{n-1}\Radon\varphi_{m}\left(\tilde{n}(y,z),s\right)\log&\abs{2\norm{z-y}(s-\tilde{s}(y,z))}\d{s}\\
				&=\lim_{\delta\searrow 0}\bigintsss_{\R\setminus(\tilde{s}-\delta,\tilde{s}+\delta)}\frac{\partial_2^{n-2}\Radon\varphi_{m}\left(\tilde{n}(y,z),s\right)}{\tilde{s}(y,z)-s}\d{s}\\
				&=\pi\Hilbert_2\partial_2^{n-2}\Radon\varphi_{m}\left(\tilde{n}(y,z),\tilde{s}(y,z)\right),
			\end{align*}
			where we set for short $\tilde{s}=\tilde{s}(y,z)$. Hence, the equalities $\gamma_n=2^{n/2}(\frac{n}{2})!$ and $\omega_n=\pi^{n/2}/\Gamma(\frac{n}{2}+1)=\pi^{n/2}/(\frac{n}{2})!$ show the above identity.

				\item  \label{item:diffhilbertradon}  Because $\Omega$ is smooth and convex, the Radon transform $\Radon\chi_\Omega$ is smooth on the set $A = \{(\theta,s)\in\Sp^{n-1}\times\R\mid E(\theta,s)\cap\Omega\neq\emptyset \}$. As the Hilbert transform is the  distributional  convolution with   the principal value distribution   $\operatorname{P.V.} 1/s$, its  Hilbert transform $\Hilbert_2\Radon\chi_\Omega$ is smooth on the set $A$, too (see, for example,  \cite{duistermaat2010distributions}).

				\item \label{item:convhilbertradon} Next, we define the set $\Omega_0\coloneqq\set{x\in\Omega\mid \mathrm{dist}(x,\partial\Omega)>\rho}$ and $A_0=\{(\theta,s)\in\Sp^{n-1}\times\R\mid E(\theta,s)\cap\Omega_0\neq\emptyset \}$, where $\rho\coloneqq\mathrm{dist}(\supp{f}\cup\supp{g},\partial\Omega)/2$. Since $\Omega$ is convex, we have that $\Omega_0$ is convex.
				
				From the properties of the Radon transform we have $\Radon \varphi_m=\Radon\chi_\Omega\ast_2\Radon\psi_m$, where we set for short $\psi_m\coloneqq\psi_{\mu,\frac{1}{m}}$. Since $\Radon\chi_\Omega(\theta,\cdot)\in L^2(\R)$, we deduce from the boundedness of the Hilbert transform $\Hilbert\colon L^2(\R)\to L^2(\R)$ and Young's inequality the relation $\Hilbert_2 \Radon \varphi_m = (\Hilbert_2\Radon\chi_\Omega)\ast_2\Radon\psi_m$. Moreover, there exists $M\in\N$ such that for all $m\geq M$ $\supp{\Radon\psi_m(\theta,\cdot)}\leq\mathrm{dist}(\overline{\Omega_0},\partial\Omega)$ for a fixed angle $\theta\in\Sp^{n-1}$. Thus, Item \ref{item:diffhilbertradon} implies that for all $m\geq M$ the function $h_\theta(\cdot,y)$ with
				\begin{equation*}
					h_\theta\colon\set{s\in\R\mid (\theta,s)\in A_0}\times\R\to\R\colon (s,y)\mapsto (\Hilbert_2\Radon\chi_\Omega)(\theta,s-y)\Radon\psi_m(\theta,y)
				\end{equation*}
				is smooth for all $y\in\R$, where
				\begin{equation*}
					\partial_2^{n-2} h_\theta(s,y)=\begin{cases}
						\partial_2^{n-2}\Hilbert_2\Radon\chi_\Omega(\theta,s-y)\Radon\psi_m(\theta,y),\quad &y\in\supp{\Radon\psi_m(\theta,\cdot)}\\
						0,\quad &\text{else}
					\end{cases}
				\end{equation*}
				and \[\abs{\partial_2^{n-2} h_\theta(s,y)}\leq \max\set{\abs{\partial_2^{n-2}\Hilbert_2\Radon\chi_\Omega(\tilde{\theta},\tilde{s})}\mid (\tilde{\theta},\tilde{s})\in \overline{A_0}}\abs{\Radon\psi_m(\theta,y)}.\] It then follows from the theorem on parametrized integrals that for all $(\theta,s)\in A_0$
				\begin{align*}
					\partial_2^{n-2}(\Hilbert_2\Radon\chi_\Omega)\ast_2\Radon\psi_m(\theta,s)&=\partial_2^{n-2}\int_\R h_\theta(s,y)\d{y}\\
					&=\int_\R \partial_2^{n-2}\Hilbert_2\Radon\chi_\Omega(\theta,s-y)\Radon\psi_m(\theta,y)\d{y}\\
					&=((\partial_2^{n-2}\Hilbert_2\Radon\chi_\Omega)\ast_2 \Radon\psi_m)(\theta,s).
				\end{align*}
				Hence, Lemma \ref{lem:radonmollifier} implies that $\partial_2^{n-2}\Hilbert_2\Radon\varphi_{m}$ converges uniformly to $\partial_2^{n-2}\Hilbert_2\Radon\chi_\Omega$ on the closed set $\overline{A_0}$. Since $(\tilde{n}(x,y),\tilde{s}(x,y))\in A_0$ for all $x,y\in\Omega_0$, we finally obtain the desired transformation from \ref{item:identity3term2even}.\qedhere
		\end{enumerate}
	\end{proof}

	\subsection{Proof of Theorem \ref{thm:explicitformeven}}
	\label{sec:proof-even}
Because of Proposition \ref{lem:transfirsttermeven} we have
		\begin{equation}
			\label{eq:thmeven1}
			\begin{aligned}
				\int_{\partial{\Omega}}&\int_0^\infty v(y,t)\nd u(y,t) \d{t}\d{\sigma(y)}\\
				&=\frac{1}{\omega_n\gamma_n}(-1)^{\frac{n-2}{2}}\int_{\partial\Omega}\int_0^\infty \int_{\R^n}\bigg(\left(\partial_t \frac{1}{t}\right)^{\frac{n-2}{2}}\nd u(y,t)\bigg)\\
				&\hspace{4.75cm}\cdot\frac{g(x)\chi_{(0,t)}(\norm{x-y})}{\sqrt{t^2-\norm{x-y}^2}}\d{x}\d{t}\d{\sigma(y)}\\ 
				&=\frac{1}{\omega_n\gamma_n}(-1)^{\frac{n-2}{2}}\int_{\Omega}g(x)\int_{\partial\Omega}\int_{\norm{x-y}}^\infty\frac{\left(\partial_t t^{-1}\right)^{\frac{n-2}{2}}\nd u(y,t)}{\sqrt{t^2-\norm{x-y}^2}}\d{t}\d{\sigma(y)}\d{x}.
			\end{aligned}
		\end{equation}
Next, we show that
			\begin{multline} \label{eq:thmeven2}
				\int_\Omega \int_0^\infty \Delta(uv)(x,t)\d{t}\d{x}
				\\ =\frac{(-1)^{\frac{n-2}{2}}}{2^{n+1}\pi^{n-1}}\int_\Omega g(x)\int_\Omega f(y)\frac{\left(\partial_2^n\Hilbert_2\Radon \chi_\Omega\right)\left(\tilde{n}(x,y),\tilde{s}(x,y)\right)}{\norm{x-y}^{n-1}}\d{y}\d{x}
			\end{multline}
			for every test function $g\in C_c^\infty(\Omega)$.
			
			Since $\Delta u$ and $\Delta v$ are solutions of the wave equation with initial data $(\Delta f,0)$ and $(0,\Delta g)$, respectively, as well as $\nabla u$ and $\nabla v$ with respect to $(\nabla f,0)$ and $(0,\nabla g)$,  from Proposition \ref{lem:transsecondtermeven} and the relation $\Delta(uv) = v\Delta u+2\scp{\nabla u,\nabla v}+u\Delta v$ it follows that
			\begin{align*}
				\int_\Omega &\int_0^\infty \Delta(uv)(x,t)\d{t}\d{x}
				\\ &= \frac{(-1)^{\frac{n-2}{2}}}{2^{n+1}\pi^{n-1}}\int_\Omega \int_\Omega\left(\Delta f(y)g(x)+2\sum_{i=1}^n \partial_if(y) \partial_ig(x) +f(y)\Delta g(x)\right)\\
&\hspace{4cm}\cdot\frac{\left(\partial_2^{n-2}\Hilbert_2\Radon \chi_\Omega\right)\left(\tilde{n}(x,y),\tilde{s}(x,y)\right)}{\norm{x-y}^{n-1}}\d{y}\d{x}.
			\end{align*}
			The application of the substitution rule on the double integral with the diffeomorphism $\Phi\colon X \to \Omega \times \Omega\colon
			(p,q)\mapsto (p,p-q)$, where $X\coloneqq\set{(x,x-y)\mid x,y\in\Omega}$, leads to
			\begin{align*}
				\int_X&\left(\Delta f(p-q)g(p)+2\sum_{i=1}^n \partial_if(p-q) \partial_ig(p) +f(p-q)\Delta g(p)\right)\\
				&\hspace{2cm}\cdot\frac{\left(\partial_2^{n-2}\Hilbert_2\Radon \chi_\Omega\right)\left(\tilde{n}(p,p-q),\tilde{s}(p,p-q)\right)}{\norm{q}^{n-1}}\d{(p,q)}\\
				&=\int_X\sum_{i=1}^n \partial_i^2(f\circ(\pi_1-\pi_2)\cdot g\circ \pi_1)(p,q)\\
				&\hspace{3cm}\cdot\frac{\left(\partial_2^{n-2}\Hilbert_2\Radon \chi_\Omega\right)\left(\tilde{n}(p,p-q),\tilde{s}(p,p-q)\right)}{\norm{q}^{n-1}}\d{(p,q)}\\
				&=\int_{\pi_2(X)}\int_{X^{[q]}}\sum_{i=1}^n \partial_i^2(f\circ(\pi_1-\pi_2)\cdot g\circ \pi_1)(p,q)\\
				&\hspace{3cm}\cdot\frac{\left(\partial_2^{n-2}\Hilbert_2\Radon \chi_\Omega\right)\left(\tilde{n}(p,p-q),\tilde{s}(p,p-q)\right)}{\norm{q}^{n-1}}\d{p}\d{q},
			\end{align*}
			where $\pi_1\colon (x,y)\mapsto x$, $\pi_2\colon (x,y)\mapsto y$ and $X^{[q]}\coloneqq\set{p\in\R^n\mid (p,q)\in X}$ for $q\in \pi_2(X)$. From partial integration and the chain rule it then follows that the double integral equals
			\begin{align*}
				&\int_{\pi_2(X)}\int_{X^{[q]}}f(p-q)g(p)\sum_{i=1}^n \partial_i^2\frac{\left(\partial_2^{n-2}\Hilbert_2\Radon \chi_\Omega\right)\left(-q/\norm{q},\tilde{s}(p,p-q)\right)}{\norm{q}^{n-1}}\d{p}\d{q}\\
				&=\int_{\pi_2(X)}\int_{X^{[q]}}f(p-q)g(p)\sum_{i=1}^n \frac{\left(\partial_2^{n}\Hilbert_2\Radon \chi_\Omega\right)\left(\tilde{n}(p,p-q),\tilde{s}(p,p-q)\right)q_i^2}{\norm{q}^{n+1}}\d{p}\d{q}\\
				&=\int_{\pi_2(X)}\int_{X^{[q]}}f(p-q)g(p)\frac{\left(\partial_2^{n}\Hilbert_2\Radon \chi_\Omega\right)\left(\tilde{n}(p,p-q),\tilde{s}(p,p-q)\right)}{\norm{q}^{n-1}}\d{p}\d{q}
			\end{align*}
			One further application of the substitution rule with $\Phi^{-1}$ shows \eqref{eq:thmeven2}.
			
			Finally, from Proposition~\ref{prop:main-id} and Equations \eqref{eq:thmeven1} and \eqref{eq:thmeven2} we have
			\begin{multline*}
				\int_\Omega f(x)g(x)\d{x}\\
				=\int_{\Omega}g(x)\left(\frac{(-1)^{\frac{n-2}{2}}}{2^{\frac{n-2}{2}}\pi^{\frac{n}{2}}}\int_{\partial\Omega}\int_{\norm{x-y}}^\infty\frac{\left(\partial_t t^{-1}\right)^{\frac{n-2}{2}}\nd u(y,t)}{\sqrt{t^2-\norm{x-y}^2}}\d{t}\d{\sigma(y)}+\K_\Omega f(x)\right)\d{x}
			\end{multline*}
			for every test function $g\in C_c^\infty(\Omega)$, which shows the claimed inversion formula in even dimension.
	
\section{Inversion in odd dimension}
\label{sec:formulaodddim}

\subsection{The inversion formulas}
\label{ssec:oddform}

The following theorem is our main result for odd dimensions.

	\begin{theorem}[Inversion formula in odd dimension]
		\label{thm:explicitformodd}
		Let $n\geq 3$ be an odd number, $\Omega\subset \R^n$ be a bounded convex domain with smooth boundary and $f\in C_c^\infty(\Omega)$. Then, for every $x\in\Omega$, we have
		\begin{equation}
			\label{eq:explictformodd}
			f(x)=\frac{1}{(2\pi)^{\frac{n-1}{2}}}(-1)^{\frac{n-3}{2}}\int_{\partial\Omega} \left(\frac{1}{t}\partial_t\right)^{\frac{n-3}{2}}\left(\frac{1}{t}\nd u\right)(y,\norm{x-y})\d{\sigma(y)}+\K_\Omega f(x),
		\end{equation}
		where
		\begin{equation*}
			\K_\Omega f (x)\coloneqq \frac{(-1)^{\frac{n-3}{2}}}{2^{n+1}\pi^{n-1}}\int_\Omega f(y)\frac{\left(\partial_2^n\Radon \chi_\Omega\right)\left(\tilde{n}(x,y),\tilde{s}(x,y)\right)}{\norm{x-y}^{n-1}}\d{y}.
		\end{equation*}
	\end{theorem}

	\begin{proof}
The proof will be given in subsection \ref{sec:proof-odd}.
\end{proof}
Similar to the even dimensional case, as a consequence  of Theorem~\ref{thm:explicitformodd}, we derive  the following exact inversion  for   the case that $\Omega$ is an elliptical domain of the form $\set{x \in \R^n \mid \norm{Q x } < 1}$ for some invertible matrix $Q \in \R^{n \times n}$.

	\begin{corollary}[Exact inversion formula for ellipsoids in odd dimension]
		\label{cor:exactformodd}
		Let $n\geq 3$ be an odd number, $\Omega\subset \R^n$ be an elliptical domain and $f\in C_c^\infty(\Omega)$. Then, for every $x\in\Omega$, we have
		\begin{equation}
			\label{eq:exactformodd}
			f(x)=\frac{1}{(2\pi)^{\frac{n-1}{2}}}(-1)^{\frac{n-3}{2}}\int_{\partial\Omega} \left(\frac{1}{t}\partial_t\right)^{\frac{n-3}{2}}\left(\frac{1}{t}\nd u\right)(y,\norm{x-y})\d{\sigma(y)}.
		\end{equation}
	\end{corollary}

	\begin{proof}
		Because of Theorem \ref{thm:explicitformodd} we are left to show that $\K_\Omega f(x)=0$ for $x\in\Omega$ as in even dimension. We refer again to \cite{Hal14}, where the identity $\partial_2^n \Radon\chi_\Omega=0$ in odd dimension has been verified.
	\end{proof}

	For the derivation of the above explicit inversion formula we follow the same strategy as in even dimension. We prove similar lemmas by using now solution formula \eqref{eq:solwaveeqodd} and make use of some parts of the proofs in even dimension. We will also see that the proof of Theorem \ref{thm:explicitformodd} is shorter than the proof in even dimension. One reason for this is that solution of the wave equation in odd dimension has a simpler form than in the even case. On the other hand, the solution of the wave equation has compact support in the time domain for each fixed point when the initial data has compact support.
	
	\subsection{Manipulation of the boundary term}
	
	As in even dimension, we start again by reshaping the first term on the right-hand side in \eqref{eq:intidentity}.
	\begin{prop}
		\label{lem:transfirsttermodd}
		Let $n\geq 3$ be an odd natural number and $f,g\in C_c^\infty(\Omega)$. Then the identity
		\begin{equation*}
			\int_0^\infty v(x,t)\nd u(x,t) \d{t}=\frac{(-1)^{\frac{n-3}{2}}}{n\omega_n\gamma_n}\int_{\R^n} g(y)\left(\frac{1}{t}\partial_t\right)^{\frac{n-3}{2}}\left(\frac{1}{t}\nd u\right)(x,\norm{x-y})\d{y}	
		\end{equation*}
		holds for every $x\in\partial\Omega$.
	\end{prop}
	\begin{proof}
		Inserting solution formula \eqref{eq:solwaveeqodd} for the function $v$ and applying integration by parts $(n-3)/2$-times lead to
		\begin{equation*}
			\int_0^\infty v(x,t)\nd u(x,t) \d{t}=\frac{(-1)^{\frac{n-3}{2}}}{n\omega_n\gamma_n}\int_0^\infty \M g(x,t)\left(\frac{1}{t}\partial_t\right)^{\frac{n-3}{2}}\left(\frac{1}{t}\nd u\right)(x,t)\d{t}.
		\end{equation*}
	Then, by using polar coordinates and the substitution $y$ with $x-y$ we obtain
	\begin{align*}
			\int_0^\infty v(x,t)\nd u(x,t) \d{t}&=\frac{(-1)^{\frac{n-3}{2}}}{n\omega_n\gamma_n}\int_{\R^n} g(x+y)\left(\frac{1}{t}\partial_t\right)^{\frac{n-3}{2}}\left(\frac{1}{t}\nd u\right)(x,\norm{y})\d{y}\\
			&=\frac{(-1)^{\frac{n-3}{2}}}{n\omega_n\gamma_n}\int_{\R^n} g(y)\left(\frac{1}{t}\partial_t\right)^{\frac{n-3}{2}}\left(\frac{1}{t}\nd u\right)(x,\norm{x-y})\d{y}.	
	\end{align*}
	\end{proof}
	
	\subsection{Manipulation of the interior term}
	
	In view of Lemma \ref{lem:identity1term2even} and \ref{lem:identity2term2even}, we analogously transform the second term in the odd case as follows:
	\begin{lemma}
		\label{lem:identity1term2odd}
		Let $n\geq 3$ be an odd natural number and $f,g\in C_c^\infty(\Omega)$. Then, we have
		\begin{equation*}
			%\label{eq:identity1term2even}
			\begin{aligned}
			\int_\Omega\int_0^\infty & u(x,t)v(x,t)\d{t}\d{x}\\
			&=\frac{(-1)^{\frac{n-1}{2}}2^{n-1}}{\gamma_n^2}\lim_{m,k\to \infty}\int_{\R^n}\int_0^\infty\int_0^\infty \varphi_{m}(x)t^{n-1}r^{n-1}\M f(x,t)\M g(x,r)\\
			&\hspace*{5.675cm}\cdot\Phi_{k}^{(n-1)}(t^2-r^2)\d{r}\d{t}\d{x}
			\end{aligned}
		\end{equation*}
		where $\set{\varphi_m}_{m\in\N}$ is a family of bounded and integrable functions converging pointwise to $\chi_{\Omega}$ and $\set{\Phi_k}_{k\in\N}$ a family of smooth functions converging pointwise to the Heaviside function $H\colon \R\to\R$ almost everywhere.
	\end{lemma}
	\begin{proof}
		From solution formula \eqref{eq:solwaveeqodd2} we deduce
		\begin{equation*}
			v(x,t)=\int_0^t \partial_r v(x,r)\d{r}=\frac{1}{\gamma_n}\int_0^\infty H(t^2-r^2)r\D_r^{\frac{n-1}{2}}(r^{n-2}\M g(x,r))\d{r}.
		\end{equation*}
		Therefore, we obtain from Lebesgue's dominated convergence theorem
		\begin{align*}
			&\int_\Omega\int_0^\infty u(x,t)v(x,t)\d{t}\d{x}\\
			&=\frac{1}{\gamma_n^2}\lim_{m,k\to\infty}\int_{\Omega}\int_0^\infty\int_0^\infty  \varphi_m(x)t\D_t^{\frac{n-1}{2}}(t^{n-2}\M f(x,t))r\D_r^{\frac{n-1}{2}}(r^{n-2}\M g(x,r))\\
			&\hspace{4cm}\cdot\Phi_k(t^2-r^2)\d{r}\d{t}\d{x}\\
			&=\frac{(-2)^{\frac{n-1}{2}}}{\gamma_n^2}\lim_{m,k\to\infty}\int_{\Omega}\int_0^\infty\int_0^\infty\varphi_m(x)t^{n-1}\M f(x,t)r\D_r^{\frac{n-1}{2}}(r^{n-2}\M g(x,r))\\
			&\hspace{4.75cm}\cdot\Phi_k^{\left(\frac{n-1}{2}\right)}(t^2-r^2)\d{r}\d{t}\d{x}\\
			&=\frac{(-1)^{\frac{n-1}{2}}2^{n-1}}{\gamma_n^2}\lim_{m,k\to\infty}\int_{\Omega}\int_0^\infty\int_0^\infty\varphi_m(x)t^{n-1}\M f(x,t)r^{n-1}\M g(x,r)\\
			&\hspace{5.5cm}\cdot\Phi_k^{(n-1)}(t^2-r^2)\d{r}\d{t}\d{x},
		\end{align*}
		where we used solution formula \eqref{eq:solwaveeqodd2} for $u$ and applied integration by parts $(n-1)/2$-times with respect to $t$ and $r$. 
	\end{proof}
	\begin{prop}
		\label{lem:transsecondtermodd}
		The left double integral in Lemma \ref{lem:identity1term2odd} can be transformed to
		\begin{multline*}
			\int_\Omega\int_0^\infty u(x,t)v(x,t)\d{t}\d{x}\\
			=\frac{(-1)^{\frac{n-3}{2}}}{2^{n+1}\pi^{n-1}}\int_\Omega g(x)\int_\Omega f(y)\frac{\left(\partial_2^{n-2}\Radon \chi_\Omega\right)\left(\tilde{n}(x,y),\tilde{s}(x,y)\right)}{\norm{x-y}^{n-1}}\d{y}\d{x}.
		\end{multline*}
	\end{prop}
	\begin{proof}
		For the proof, we use again the approximation of identity $\{\psi_{\mu,\frac{1}{m}}\}_{m\in\N}$ from Lemma \ref{lem:radonmollifier} with $\mu=\frac{n}{2}+1$ and set again $\varphi_m\coloneqq \chi_{\Omega}\ast \psi_{\mu,\frac{1}{m}}$ for $m\in\N$.

		From Lemma \ref{lem:identity1term2odd} we obtain
		\begin{align*}
			\int_\Omega\int_0^\infty & u(x,t)v(x,t)\d{t}\d{x}\\
			&=\frac{(-1)^{\frac{n-1}{2}}2^{n-1}}{n^2\omega_n^2\gamma_n^2}\lim_{m,k\to \infty}\int_{\R^n}\int_\Omega \int_\Omega g(y)f(z)\varphi_{m}(x)\\
				&\hspace{5.25cm}\cdot\Phi_{k}^{(n-1)}(\norm{x-y}^2-\norm{x-z}^2)\d{z}\d{y}\d{x}
		\end{align*}
		by using polar coordinates. Then the same arguments as in the proof of Proposition \ref{lem:transsecondtermeven} give
		\begin{align*}
			\int_\Omega&\int_0^\infty u(x,t)v(x,t)\d{t}\d{x}\\
			&=\frac{(-1)^{\frac{n-1}{2}}}{n^2\omega_n^2\gamma_n^2}\lim_{m\to \infty}\int_\Omega g(y)\int_\Omega \frac{f(z)}{\norm{z-y}^{n-1}}\\
			&\hspace{3cm}\cdot\int_{\R}\partial_2^{n-1}\Radon\varphi_{m}\left(\tilde{n}(y,z),s\right)H(2\norm{z-y}(s-\tilde{s}(y,z)))\d{s}\d{z}\d{y},
		\end{align*}
		where the inner integral can be evaluated to
		\begin{equation*}
			\int_{\tilde{s}(y,z)}^\infty\partial_2^{n-1}\Radon\varphi_{m}\left(\tilde{n}(y,z),s\right)\d{s}=-\partial_2^{n-2}\Radon\varphi_{m}\left(\tilde{n}(y,z),\tilde{s}(y,z)\right).
		\end{equation*}
		Finally, we can use the proof of Proposition \ref{lem:transsecondtermeven} \ref{item:convhilbertradon} to deduce the desired transformation.
	\end{proof}
	
	\subsection{Proof of Theorem \ref{thm:explicitformodd}}
	\label{sec:proof-odd}

		As a first step, we see that the first term in \eqref{eq:intidentity} equals
		\begin{multline*}
			\int_{\partial\Omega}\int_0^\infty v(x,t)\nd u(x,t) \d{t}\\
			=\frac{1}{2^{\frac{n+1}{2}}\pi^{\frac{n-1}{2}}}(-1)^{\frac{n-3}{2}}\int_{\Omega}g(x)\int_{\partial\Omega} \left(\frac{1}{t}\partial_t\right)^{\frac{n-3}{2}}\left(\frac{1}{t}\nd u\right)(y,\norm{x-y})\d{\sigma(y)}\d{x}
		\end{multline*}
		by changing the order of integration. Next, we apply the substitution rule with the diffeomorphism $\Phi$ as in the proof of Theorem \ref{thm:explicitformeven} \eqref{eq:thmeven2} and partial integration on the right integral in Proposition \ref{lem:transsecondtermodd} to obtain
		\begin{multline*}
			\int_\Omega \int_0^\infty \Delta(uv)(x,t)\d{t}\d{x}\\
			=\frac{(-1)^{\frac{n-3}{2}}}{2^{n+1}\pi^{n-1}}\int_\Omega g(x)\int_\Omega f(y)\frac{\left(\partial_2^n\Radon \chi_\Omega\right)\left(\tilde{n}(x,y),\tilde{s}(x,y)\right)}{\norm{x-y}^{n-1}}\d{y}\d{x}.
		\end{multline*}
		Hence we have
		\begin{multline*}
			\int_\Omega f(x)g(x)\\
			=\int_{\Omega}g(x)\left(\frac{(-1)^{\frac{n-3}{2}}}{(2\pi)^{\frac{n-1}{2}}}\int_{\partial\Omega} \left(\frac{1}{t}\partial_t\right)^{\frac{n-3}{2}}\left(\frac{1}{t}\nd u\right)(y,\norm{x-y})\d{\sigma(y)}+\K_\Omega f(x)\right)\d{x}
		\end{multline*}
		for every test function $g\in C_c^\infty(\Omega)$, which shows the claimed inversion formula in odd dimension.

	\section{Conclusion} \label{sec:conclusion}
	
	In this article we studied the problem of determining the initial data of the wave equation from Neumann traces in arbitrary dimension. This problem is particularly interesting in PAT, where the aim is to recover the initial pressure distribution $f\colon\R^n\to\R$ from measurements on some boundary enclosing the unknown object. We derived explicit inversion formulas for Neumann measurements on smooth boundaries of convex domains up to an additional integral operator depending on the unknown function $f$. As we have seen, this integral operator vanishes for elliptical domains which results in exact reconstruction.
	
	By taking a closer look to the inversion formulas \eqref{eq:explictformeven} and \eqref{eq:explictformodd}, we observe that the formula for odd dimension requires only knowledge of Neumann traces on a finite time interval. This follows simply from the fact that the distance between two points inside a bounded domains is always smaller than its diameter. However, in even dimensions, the inversion formula requires knowledge of the Neumann trace for all positive times. In practice, only measurements on a finite time interval are known. Therefore, we intend to investigate inversion formulas that only require data on a finite time interval in even dimensions as well. A possible approach to this issue is similar to that mentioned in \cite{FinHalRak07} for Dirichlet traces, where relations between the solution of the wave equation and spherical means have been employed.
	
	In the case of two spatial dimension, in \cite{DreHal20} we derive exact inversion formulas for any linear combination of the solution of wave equation and its normal derivative on circular domains. To study an analogous problem in higher dimensions is an interesting open issue. For that purpose, the derivation of so-called range conditions for the wave equation \cite{agranovsky3range,ambartsoumian2006range,finch2017recovering} could be helpful. Another subject worth studying more profoundly is the development of exact inversion formulas for other special domains.

Finally, note that the presented inversion formulas clearly imply uniqueness of recovering the initial data of the wave equation from Neuman traces on ellipsoids. For more general domains, such a conclusion would follow by proving the invertibility of the operator $I-\K_{\Omega}$.  From \cite{Hal14,nguyen2014reconstruction}, where the same operator appears, it follows that $I-\K_{\Omega}$ is a Fredholm operator. However, we could not prove that $I-\K_{\Omega}$ has a zero kernel. Investigating invertibility for general domains in both the complete and partial data situations is an interesting line for future research.

\appendix

\section{Remaining proofs}

\subsection{Proof of Lemma \ref{lem:radonmollifier} }

As a first step, we have to show \[\int_{\B^n(0,1)}(1-\norm{x}^2)^\mu \d{x}=\frac{1}{a}.\]
			By using polar coordinates and letting $r=\sqrt{u}$ we see that
			\begin{equation*}
				\int_{\B^n(0,1)}(1-\norm{x}^2)^\mu \d{x}=n\omega_n\int_0^1(1-r^2)^\mu r^{n-1}\d{r}=\frac{n\omega_n}{2}\int_0^1(1-u)^\mu u^{\frac{n-2}{2}}\d{u},
			\end{equation*}
			where $\omega_n\coloneqq\vol(\B^n(0,1))$ denotes the volume of the $n$-dimensional unit ball. Then, applying integration by parts $\mu$-times yields
			\begin{align*}
				\int_{\B^n(0,1)}(1-\norm{x}^2)^\mu \d{x}&=\frac{n\omega_n}{2}\frac{\mu!}{\left(\frac{n-2}{2}+1\right)\cdots\left(\frac{n-2}{2}+\mu+1\right)}\\
				&=\frac{n\omega_n}{2}\frac{\Gamma(\mu+1)\Gamma(\frac{n}{2})}{\Gamma(\frac{n}{2}+\mu+1)}=\frac{1}{a},
			\end{align*}
			where we used the identities $\omega_n=\pi^{n/2}/\Gamma(\frac{n}{2}+1)$ and $\Gamma(\frac{n}{2}+1)=\frac{n}{2}\Gamma(\frac{n}{2})$.

To prove the second statement, we first show \eqref{eq:radonmollifier}. By using polar coordinates and letting $r=\sqrt{u}$ again we have
			\begin{align*}
				\Radon\psi_{\mu,\varepsilon}(\theta,s)&=\frac{1}{\varepsilon^n a}\int_{\B^{n-1}(0,\sqrt{\varepsilon^2-s^2})}\bigg(1-\frac{\lVert s\theta+\sum_{i=1}^{n-1}x_i\theta_i\rVert^2}{\varepsilon^2}\bigg)^\nu\d{x}\\
				&=\frac{1}{\varepsilon^n a}\int_{\B^{n-1}(0,\sqrt{\varepsilon^2-s^2})}\bigg(1-\frac{s^2+\norm{x}^2}{\varepsilon^2}\bigg)^\nu\d{x}\\
				&=\frac{(n-1)\omega_{n-1}}{\varepsilon^n a}\int_0^{\sqrt{\varepsilon^2-s^2}}\bigg(1-\frac{s^2+r^2}{\varepsilon^2}\bigg)^\nu r^{n-2}\d{r}\\
				&=\frac{(n-1)\omega_{n-1}}{2\varepsilon^n a}\int_0^{\varepsilon^2-s^2}\bigg(1-\frac{s^2+u^2}{\varepsilon^2}\bigg)^\nu u^{\frac{n-3}{2}}\d{u},
			\end{align*}
			where $(\theta_1,\ldots,\theta_{n-1})$ is a orthonormal basis of $E(\theta,s)$. As in the first step, applying integration by parts $\mu$-times and using the same identities as before lead to
			\begin{align*}
				\Radon\psi_{\mu,\varepsilon}(\theta,s)&=\frac{(n-1)\omega_{n-1}}{2\varepsilon^n a}\frac{\Gamma(\mu+1)\Gamma(\frac{n-1}{2})}{\varepsilon^{2\mu}\Gamma(\frac{n-1}{2}+\mu+1)}\left(\varepsilon^2-s^2\right)^{\frac{n-3}{2}+\mu+1}\\
				&=\frac{\Gamma\left(\frac{n}{2}+\mu+1\right)}{\varepsilon\sqrt{\pi}\Gamma\left(\frac{n-1}{2}+\mu+1\right)}\left(1-\frac{s^2}{\varepsilon^2}\right)^{\frac{n-3}{2}+\mu+1}.
			\end{align*}
			Now, it remains to show $\int_\R \Radon\psi_{\mu,\varepsilon}(\theta,s)\d{s}=1$. By letting $u=s/\varepsilon$, $x=\cos(u)$ and using the well-known relation
			\begin{equation*}
				\int_{-\pi/2}^{\pi/2}\cos(x)^m\d{x}=\frac{\sqrt{\pi}\Gamma(\frac{m+1}{2})}{\Gamma(\frac{m}{2}+1)},\quad m\in\N,
			\end{equation*}
			we finally have
			\begin{align*}
				\int_\R \Radon\psi_{\mu,\varepsilon}(\theta,s)\d{s}&=\frac{\Gamma\left(\frac{n}{2}+\mu+1\right)}{\sqrt{\pi}\Gamma\left(\frac{n-1}{2}+\mu+1\right)}\int_{-\varepsilon}^\varepsilon \left(1-\frac{s^2}{\varepsilon^2}\right)^{\frac{n-3}{2}+\mu+1}\frac{1}{\varepsilon}\d{s}\\
				&=\frac{\Gamma\left(\frac{n}{2}+\mu+1\right)}{\sqrt{\pi}\Gamma\left(\frac{n-1}{2}+\mu+1\right)}\int_{-1}^1 \left(1-u^2\right)^{\frac{n-3}{2}+\mu+1}\d{u}\\
				&=\frac{\Gamma\left(\frac{n}{2}+\mu+1\right)}{\sqrt{\pi}\Gamma\left(\frac{n-1}{2}+\mu+1\right)}\int_{-\pi/2}^{\pi/2}\cos(x)^{n+2\mu}\d{x}=1.
			\end{align*}

\subsection{Proof of Lemma \ref{lem:wave-even}}

First, we show formula \eqref{eq:solwaveeqeven2}. Let $f_n\colon\R^n\to\R^n\colon (r,\varphi,\theta_1,\theta_2,\ldots,\theta_{n-2})\mapsto (x_1,\ldots,x_n)$ be the $n$-dimensional polar coordinate map, where
			\begin{align*}
				x_1&=r\cos\varphi\sin\theta_1\sin\theta_2\cdots\sin\theta_{n-2},\\
				x_2&=r\sin\varphi\sin\theta_1\sin\theta_2\cdots\sin\theta_{n-2},\\
				x_3&=r\cos\theta_1\sin\theta_2\cdots\sin\theta_{n-2},\\
				&\vdots\\
				x_{n-1}&=r\cos\theta_{n-3}\sin\theta_{n-2},\\
				x_n&=r\cos\theta_{n-2},
			\end{align*}
			and $g_n\colon(0,2\pi)\times(0,\pi)^{n-2}\to\R^n\colon (\varphi,\theta_1,\theta_2,\ldots,\theta_{n-2})\mapsto f_n(1,\varphi,\theta_1,\theta_2,\ldots,\theta_{n-2})$ a parametrization of $\Sp^{n-1}$. Then, for a function $h\in C_c^\infty(\Omega)$ and $(x,t)\in\R^n\times(0,\infty)$ we obtain
			\begin{align*}
				\int_{\B^n(x,t)}&\frac{h(y)}{\sqrt{t^2-\norm{y-x}^2}}\d{y}\\
				&=\int_{\R^n}\frac{\chi_{\B^n(0,t)}(y)h(x+y)}{\sqrt{t^2-\norm{y}^2}}\d{y}\\
				&=\int_0^\infty\int_{(0,2\pi)\times(0,\pi)^{n-2}}\frac{\chi_{\B^1(0,t)}(r)h(x+f_n(r,\varphi,\theta))}{\sqrt{t^2-r^2}}\abs{\det \partial f_n(r,\varphi,\theta)}\d{(\varphi,\theta)}\d{r}
			\end{align*}
			by substituting $y$ with $y+x$ and using polar coordinates. Since the Gram determinant of the parametrization $x+rg_n$ of $\partial\B^n(x,t)$ at $(\varphi,\theta)\in(0,2\pi)\times(0,\pi)^{n-2}$ equals $\abs{\det \partial f_n(r,\varphi,\theta)}$ for every $r>0$, we have
			\begin{align*}
				\int_{\B^n(x,t)}\frac{h(y)}{\sqrt{t^2-\norm{y-x}^2}}\d{y}&=\int_0^\infty \frac{\chi_{\B^1(0,t)}(r)}{\sqrt{t^2-r^2}}\int_{\partial\B^n(x,r)}h(y)\d{\sigma(y)}\d{r}\\
				&=\int_0^t\sigma(\partial\B^n(x,r))\frac{\M h(x,r)}{\sqrt{t^2-r^2}}\d{r}.
			\end{align*}
			Inserting this identity into \eqref{eq:solwaveeqeven} by replacing $h$ with $f$ and $g$, and using the relations $\vol(\B^n(x,t))=t^n\omega_n$ and $\sigma(\partial\B^n(x,r))=r^{n-1}n\omega_n$ lead to Equality \eqref{eq:solwaveeqeven2}.
	
	Now, we show formula \eqref{eq:solwaveeqeven3}. According to \eqref{eq:solwaveeqeven2}, we are left to show that
			\begin{multline}
				\label{eq:identity1}
				\left(\frac{1}{t}\partial_t\right)^{\frac{n-2}{2}}\left(\int_0^t\frac{r^{n-1}}{\sqrt{t^2-r^2}}\M h(x,r)\d{r}\right)\\
				=\left(\int_0^t\frac{r}{\sqrt{t^2-r^2}}\left(\frac{1}{r}\partial_r\right)^{\frac{n-2}{2}}\left(r^{n-2}\M h(x,r)\right)\d{r}\right)
			\end{multline}
			for $h\in C_c^\infty(\Omega)$. First, we observe that application of integration by parts yields
			\begin{equation*}
				\left(\frac{1}{t}\partial_t\right)\left(\int_0^t\frac{r^{n-1}}{\sqrt{t^2-r^2}}\M h(x,r)\d{r}\right)\\
				=\int_0^t\sqrt{t^2-r^2}\partial_r\left(r^{n-2}\M h(x,r)\right)\d{r}
			\end{equation*}
			Then the Leibniz rule for integrals gives us
			\begin{equation*}
				\left(\frac{1}{t}\partial_t\right)\left(\int_0^t\frac{r^{n-1}}{\sqrt{t^2-r^2}}\M h(x,r)\d{r}\right)=\int_0^t\frac{r}{\sqrt{t^2-r^2}}\left(\frac{1}{r}\partial_r\right)\left(r^{n-2}\M h(x,r)\right)\d{r}.
			\end{equation*}
			Now, suppose that
			\begin{equation*}
				\left(\frac{1}{t}\partial_t\right)^{k}\left(\int_0^t\frac{r^{n-1}}{\sqrt{t^2-r^2}}\M h(x,r)\d{r}\right)\\
				=\int_0^t\frac{r}{\sqrt{t^2-r^2}}\left(\frac{1}{r}\partial_r\right)^{k}\left(r^{n-2}\M h(x,r)\right)\d{r}
			\end{equation*}
			holds for any value $k<\frac{n-2}{2}$. Since $$\lim_{r\searrow 0}\left(\frac{1}{r}\partial_r\right)^k\left(r^{n-2}\M h(x,r)\right)=0,$$
			we obtain from our assumption and partial integration
			\begin{align*}
				\left(\frac{1}{t}\partial_t\right)^{k+1}&\left(\int_0^t\frac{r^{n-1}}{\sqrt{t^2-r^2}}\M h(x,r)\d{r}\right)\\
				&=\left(\frac{1}{t}\partial_t\right)\left(\int_0^t\frac{r}{\sqrt{t^2-r^2}}\left(\frac{1}{r}\partial_r\right)^k\left(r^{n-2}\M h(x,r)\right)\d{r}\right)\\
				&=\left(\frac{1}{t}\partial_t\right)\left(\int_0^t\sqrt{t^2-r^2}\,\partial_r\left(\frac{1}{r}\partial_r\right)^k\left(r^{n-2}\M h(x,r)\right)\right)
			\end{align*}	
			Hence, another application of the Leibniz rule for integrals implies \eqref{eq:identity1}.

%	\begin{small}
%		\addcontentsline{toc}{section}{References}
%		\bibliographystyle{plain}
%		%\bibliographystyle{abbrv}
%		\bibliography{references}

\begin{thebibliography}{10}

\bibitem{acosta2019well}
{S.~Acosta}, {\em Well-posedness for photoacoustic tomography with
  fabry--perot sensors}, SIAM J. Imaging Sci., 12 (2019), pp.~1669--1685.

\bibitem{acosta2020solvability}
{S.~Acosta}, {\em Solvability for photoacoustic imaging with idealized
  piezoelectric sensors}, arXiv:2002.09929,  (2020).

\bibitem{agranovsky3range}
{M.~Agranovsky, D.~Finch, and P.~Kuchment}, {\em Range conditions for a
  spherical mean transform}, Inverse Probl. Imaging, 3 (2009), pp.~373--382.

\bibitem{ambartsoumian2006range}
{G.~Ambartsoumian and P.~Kuchment}, {\em A range description for the planar
  circular {R}adon transform}, SIAM J. Math. Anal., 38 (2006), pp.~681--692.

\bibitem{And88}
{L.-E. Andersson}, {\em On the determination of a function from spherical
  averages}, SIAM J. Math. Anal., 19 (1988), pp.~214--232.

\bibitem{AnsFibMadSey12}
{M.~Ansorg, F.~Filbir, W.~R. Madych, and R.~Seyfried}, {\em Summability
  kernels for circular and spherical mean data}, Inverse Problems, 29 (2012),
  p.~015002.

\bibitem{Bel09}
{A.~Beltukov}, {\em Inversion of the spherical mean transform with sources
  on a hyperplane}, arXiv preprint arXiv:0910.1380,  (2009).

\bibitem{BleCohJoh13}
{N.~Bleistein, J.~K. Cohen, J.~W. Stockwell~Jr., et~al.}, {\em Mathematics
  of multidimensional seismic imaging, migration, and inversion}, vol.~13,
  Springer Science \& Business Media, 2013.

\bibitem{BukKar78}
{A.~L. Bukhgeim and V.~B. Kardakov}, {\em Solution of the inverse problem
  for the equation of elastic waves by the method of spherical means}, Siberian
  Math. J., 19 (1978), pp.~528--535.

\bibitem{cox2007frequency}
{B.~T. Cox and P.~C. Beard}, {\em The frequency-dependent directivity of a
  planar {Fabry-Perot} polymer film ultrasound sensor}, IEEE Trans. Ultrason.,
  Ferroelectr., Freq. Control, 54 (2007), pp.~394--404.

\bibitem{DoKun18}
{N.~Do and L.~A. Kunyansky}, {\em Theoretically exact photoacoustic
  reconstruction from spatially and temporally reduced data}, Inverse Problems,
  34 (2018), p.~094004.

\bibitem{DreHal20}
{F.~Dreier and M.~Haltmeier}, {\em Explicit inversion formulas for the
  two-dimensional wave equation from {N}eumann traces}, SIAM J. Imaging Sci.,
  13 (2020), pp.~589--608.

\bibitem{duistermaat2010distributions}
{J.~J. Duistermaat and J.~Kolk}, {\em Distributions: Theory and
  Applications}, Birkhäuser, 2010.

\bibitem{Eva10}
{L.~C. Evans.}, {\em Partial Differential Equations}, American Mathematical
  Society, Second Edition, 2010.

\bibitem{Faw85}
{J.~A. Fawcett}, {\em Inversion of n-dimensional spherical averages}, SIAM
  J. Appl. Math., 45 (1985), pp.~336--341.

\bibitem{Finch2005}
{D.~Finch}, {\em On a thermoacoustic transform}, Proc. 8th Int. Meeting on
  Fully 3D Image Reconstruction in Radiology and Nuclear Medicine,  (2005),
  pp.~150--151.

\bibitem{FinHalRak07}
{D.~Finch, M.~Haltmeier, and Rakesh}, {\em Inversion of spherical means and
  the wave equation in even dimensions}, SIAM J. Appl. Math., 68 (2007),
  pp.~392--412.

\bibitem{FinPatRak04}
{D.~Finch, S.~K. Patch, and Rakesh}, {\em Determining a function from its
  mean values over a family of spheres}, SIAM J. Math. Anal., 35 (2004),
  pp.~1213--1240.

\bibitem{finch2017recovering}
{D.~Finch and Rakesh}, {\em Recovering a function from its spherical mean
  values in two and three dimensions}, in Photoacoustic imaging and
  spectroscopy, CRC Press, 2017, pp.~77--88.

\bibitem{Hal13}
{M.~Haltmeier}, {\em Inversion of circular means and the wave equation on
  convex planar domains}, Comput. Math. Appl., 65 (2013), pp.~1025--1036.

\bibitem{Hal14}
{M.~Haltmeier}, {\em Universal inversion formulas for recovering a function
  from spherical means}, SIAM J. Math. Anal., 46 (2014), pp.~214--232.

\bibitem{HalSer15a}
{M.~Haltmeier and S.~Pereverzyev~Jr}, {\em Recovering a function from
  circular means or wave data on the boundary of parabolic domains}, SIAM J.
  Imaging Sci., 8 (2015), pp.~592--610.

\bibitem{HalSer15b}
{M.~Haltmeier and S.~Pereverzyev~Jr}, {\em The universal back-projection
  formula for spherical means and the wave equation on certain quadric
  hypersurfaces}, J. Math. Anal. App, 429 (2015), pp.~366--382.

\bibitem{KucKun11}
{P.~Kuchment and L.~A. Kunyansky}, {\em Mathematics of photoacoustic and
  thermoacoustic tomography}, in Handbook of Mathematical Methods in Imaging,
  Springer, 2011, pp.~817--865.

\bibitem{Kun07}
{L.~A. Kunyansky}, {\em Explicit inversion formulae for the spherical mean
  {R}adon transform}, Inverse Problems, 23 (2007), p.~373.

\bibitem{Kun11}
{L.~A. Kunyansky}, {\em Reconstruction of a function from its spherical
  (circular) means with the centers lying on the surface of certain polygons
  and polyhedra}, Inverse Problems, 27 (2011), p.~025012.

\bibitem{Kun15}
{L.~A. Kunyansky}, {\em Inversion of the spherical means transform in
  corner-like domains by reduction to the classical {R}adon transform}, Inverse
  Problems, 31 (2015), p.~095001.

\bibitem{LasLioTri86}
{I.~Lasiecka, J.-L. Lions, and R.~Triggiani}, {\em Non homogeneous boundary
  value problems for second order hyperbolic operators}, J. Math. Pures Appl.,
  65 (1986), pp.~149--192.

\bibitem{LidDurLitHua09}
{C.~Li, N.~Duric, P.~Littrup, and L.~Huang}, {\em In vivo breast
  sound-speed imaging with ultrasound tomography}, Ultrasound Med. Biol., 35
  (2009), pp.~1615--1628.

\bibitem{NarRak10}
{E.~K. Narayanan and Rakesh}, {\em Spherical means with centers on a
  hyperplane in even dimensions}, Inverse Problems, 26 (2010), p.~035014.

\bibitem{Nat12}
{F.~Natterer}, {\em Photo-acoustic inversion in convex domains}, Inverse
  Probl. Imaging, 6 (2012), pp.~1--6.

\bibitem{NatWub95}
{F.~Natterer and F.~Wubbeling}, {\em A propagation-backpropagation method
  for ultrasound tomography}, Inverse Problems, 11 (1995), p.~1225.

\bibitem{Ngu09}
{L.~Nguyen}, {\em A family of inversion formulas in thermoacoustic
  tomography}, Inverse Probl. Imaging, 3 (2009), pp.~649--675.

\bibitem{nguyen2014reconstruction}
{L.~V. Nguyen}, {\em On a reconstruction formula for spherical {r}adon
  transform: a microlocal analytic point of view}, Anal. Math. Phys., 4 (2014),
  pp.~199--220.

\bibitem{NorLin81}
{S.~J. Norton and M.~Linzer}, {\em Ultrasonic reflectivity imaging in three
  dimensions: exact inverse scattering solutions for plane, cylindrical, and
  spherical apertures}, IEEE Trans. Biomed. Eng.,  (1981), pp.~202--220.

\bibitem{Pal14}
{V.~P. Palamodov}, {\em Time reversal in photoacoustic tomography and
  levitation in a cavity}, Inverse Problems, 30 (2014), p.~125006.

\bibitem{paltauf2009characterization}
{G.~Paltauf, R.~Nuster, and P.~Burgholzer}, {\em Characterization of
  integrating ultrasound detectors for photoacoustic tomography}, J. Appl.
  Phys., 105 (2009), p.~102026.

\bibitem{poudel2019survey}
{J.~Poudel, Y.~Lou, and M.~A. Anastasio}, {\em A survey of computational
  frameworks for solving the acoustic inverse problem in three-dimensional
  photoacoustic computed tomography}, Phys. Med. Biol., 64 (2019), p.~14TR01.

\bibitem{QuiRieSch11}
{E.~T. Quinto, A.~Rieder, and T.~Schuster}, {\em Local inversion of the
  sonar transform regularized by the approximate inverse}, Inverse Problems, 27
  (2011), p.~035006.

\bibitem{rosenthal2013acoustic}
{A.~Rosenthal, V.~Ntziachristos, and D.~Razansky}, {\em Acoustic inversion
  in optoacoustic tomography: A review}, Curr. Med. Imaging Rev., 9 (2013),
  pp.~318--336.

\bibitem{Sal14}
{Y.~Salman}, {\em An inversion formula for the spherical mean transform
  with data on an ellipsoid in two and three dimensions}, J. Math. Anal. App,
  420 (2014), pp.~612--620.

\bibitem{WaHu12}
{L.~V. Wang and S.~Hu}, {\em Photoacoustic tomography: In vivo imaging from
  organelles to organs}, Science, 335 (2012), pp.~1458--1462.

\bibitem{wang2007boundary}
{L.~V. Wang and X.~Yang}, {\em Boundary conditions in photoacoustic
  tomography and image reconstruction}, J. Biomed. Opt., 12 (2007), p.~014027.

\bibitem{WaPaKuXiStWa03}
{X.~Wang, Y.~Pang, G.~Ku, X.~Xie, G.~Stoica, and L.~V. Wang}, {\em
  Noninvasive laser-induced photoacoustic tomography for structural and
  functional in vivo imaging of the brain}, Nat. Biotechnol., 21 (2003),
  pp.~803--806.

\bibitem{WisPleRosNtz18}
{G.~Wissmeyer, M.~A. Pleitez, A.~Rosenthal, and V.~Ntziachristos}, {\em
  Looking at sound: optoacoustics with all-optical ultrasound detection}, Light
  Sci. Appl., 7 (2018), pp.~1--16.

\bibitem{xu2005universal}
{M.~Xu and L.~V. Wang}, {\em Universal back-projection algorithm for
  photoacoustic computed tomography}, Phys. Rev. E, 71 (2005), p.~016706.

\bibitem{zangerl2018photoacoustic}
{G.~Zangerl, S.~Moon, and M.~Haltmeier}, {\em Photoacoustic tomography with
  direction dependent data: An exact series reconstruction approach}, Inverse
  Problems, 35 (2019), p.~114005.

\end{thebibliography}
%	\end{small}
\end{document}